\documentclass[pdflatex,sn-mathphys-num]{sn-jnl}% Math and Physical Sciences Numbered Reference Style
\usepackage{graphicx}%
\usepackage{multirow}%
\usepackage{amsmath,amssymb,amsfonts}%
\usepackage{amsthm}%
\usepackage{mathrsfs}%
\usepackage[title]{appendix}%
\usepackage{xcolor}%
\usepackage{textcomp}%
\usepackage{manyfoot}%
\usepackage{booktabs}%
\usepackage{algorithm}%
\usepackage{algorithmicx}%
\usepackage{algpseudocode}%
\usepackage{listings}%

\usepackage{mathtools}
\usepackage{cleveref}

\usepackage{enumitem}

\theoremstyle{thmstyleone}%
\newtheorem{theorem}{Theorem}%  meant for continuous numbers
\newtheorem{proposition}[theorem]{Proposition}% 

\theoremstyle{thmstyletwo}%
\newtheorem{remark}{Remark}%
\newtheorem{exmp}{Example}[section]
\theoremstyle{thmstylethree}%
\newtheorem{definition}{Definition}%
\newtheorem{assumption}{Assumption}%

\begin{document}

\title[Article Title]{An Improved Boosted DC Algorithm for Nonsmooth Functions with Applications in Image Recovery}


\author[1]{\fnm{Zeyu} \sur{Li}}\email{1155184076@link.cuhk.edu.hk}
\equalcont{These authors contributed equally to this work.}
\author[1]{\fnm{Te} \sur{Qi}}\email{tqi@math.cuhk.edu.hk}
\equalcont{These authors contributed equally to this work.}

\author*[2,3]{\fnm{Tieyong} \sur{Zeng}}\email{tieyongzeng@bnbu.edu.cn}

\affil*[1]{\orgdiv{Department of Mathematics}, \orgname{Chinese University of Hong Kong}, \orgaddress{\street{Shatin}, \postcode{999077}, \city{Hong Kong}}}

\affil[2]{\orgdiv{Institute for Advanced Study}, \orgname{Beijing Normal Hong Kong Baptist University}, \orgaddress{ \city{Zhuhai},  \state{Guangdong},\postcode{519087}, \country{China}}}

\affil[3]{\orgdiv{School of Mathematics and Statistics}, \orgname{Guangzhou Nanfang College}, \orgaddress{\city{Guangzhou}, \postcode{510970}, \state{Guangdong}, \country{China}}}


\abstract{We propose a new approach to perform the boosted difference of convex functions algorithm (BDCA) on non-smooth and non-convex problems involving the difference of convex (DC) functions. The recently proposed BDCA uses an extrapolation step from the point computed by the classical DC algorithm (DCA) via a line search procedure in a descent direction to get an additional decrease of the objective function and accelerate the convergence of DCA. However, when the first function in DC decomposition is non-smooth, the direction computed by BDCA can be ascent and a monotone line search cannot be performed. In this work, we proposed a monotone improved boosted difference of convex functions algorithm (IBDCA) for certain types of non-smooth DC programs, namely those that can be formulated as the difference of a possibly non-smooth function and a smooth one. We show that any cluster point of the sequence generated by IBDCA is a critical point of the problem under consideration and that the corresponding objective value is monotonically decreasing and convergent. We also present the global convergence and the convergent rate under the Kurdyka–Łojasiewicz property. The applications of IBDCA in image recovery show the effectiveness of our proposed method. The corresponding numerical experiments demonstrate that our IBDCA outperforms DCA and other state-of-the-art DC methods in both computational time and number of iterations.}




\maketitle

\section{Introduction}
In this work, we mainly focus on the difference of convex (DC) optimization problem

\begin{equation}
\min_{x\in\mathbb{R}^m} \phi(x) \coloneqq g(x)-h(x), \label{DCproblem}
\end{equation}
where $g,h: \mathbb{R}^m\to\mathbb{R}\cup\left \{ +\infty \right \}$ are proper convex functions, and $\phi(x)$ satisfies:
\begin{equation}
\inf_{x \in \mathbb{R}^{m}}\phi(x)> -\infty.
\end{equation}
Here, we also make conventions on the definition of $\infty$ as follows \cite{nonsmoothBDCA,linearBDCA}:
\begin{gather*}
(+\infty)-(+\infty)=(+\infty)-(+\infty)=+\infty,
\\
(+\infty)-c=+\infty \qquad\mathrm{and}\qquad c-(+\infty)=-\infty,\qquad \forall c\in\left ] \infty,+\infty  \right [. 
\end{gather*}

The well-known difference of convex functions algorithm (DCA) \cite{AN1996DCA,TAO1986DCA} is one of the most efficient algorithms to solve \eqref{DCproblem}. Extensive research and development on DCA have spanned more than 30 years \cite{LE2018DCAthirty, TAO1997ConvexDC,Thi2023Open}. The DCA has found widespread applications in various fields, including machine learning, sparse optimization, deep learning and image processing \cite{Geremew2018, Gotoh2018SparseDC, LZOX2015,CDiNN2022}. Over the years, several variants of the DCA have been explored both theoretically and practically. Notable variants of the DCA include the subgradient method \cite{Ferreira2021boosted, Bagirov2021Augmented}, proximal algorithms \cite{Cruz2020generalized,Souza2016global,li2017quadratic}, and inertial methods \cite{De2019inertial}. 

Recently, \cite{aragon2018accelerating,nonsmoothBDCA} proposed a boosted DC algorithm (BDCA) to solve \eqref{DCproblem} and significantly accelerate the convergence of DCA. Among them, \cite{aragon2018accelerating} focuses on the case when both $g$ and $h$ in \eqref{DCproblem} are smooth. \cite{nonsmoothBDCA} later extended to the case where $g$ is smooth and $h$ is possibly non-smooth.  
% as well as linearly constrained DC programs \cite{linearBDCA}. 
BDCA performs an additional line search step in a descent direction based on the points computed by DCA. In addition to speeding up the convergence speed of DCA, BDCA can also effectively prevent the algorithm from falling into local optima due to the line search step, thus increasing the possibility of converging to the global optimal solution\cite[Example 3.1]{nonsmoothBDCA}.

A specific type of DC problem, where the first component in the DC decomposition may be non-smooth while the second is smooth, frequently appears in various models in the fields of image and signal processing, including applications such as image denoising \cite{Li2016variational,Sciacchitano2015variational,CHEN2018Sparse,WU2017Robust}, non-convex regularization \cite{Zhang2020tv,Gu2017tvscad,Lanza2016convex}, compressed sensing \cite{Nguyen2014dc}, phase retrieval \cite{Huang2019dc}, and dictionary learning \cite{Li2021fast}. However, when the function $g$ is non-smooth, the search direction in the BDCA can lead to ascent directions (see \cite[Example 3.2]{nonsmoothBDCA} and \cite[Remark 3.1]{aragon2018accelerating}). Consequently, BDCA may not effectively address this type of DC problem. To address this issue, Ferreira et al. \cite{ferreira2021nmBDCA} introduced a non-monotone boosted difference of convex functions algorithm (nmBDCA) and further applied nmBDCA to image denoising \cite{FERREIRA2024Image}. This algorithm employs a non-monotone line search and is applicable when both components of the DC function are non-smooth. The nmBDCA permits controlled growth in the objective function values, managed by specific parameters during the non-monotone line search step. Several other types of accelerated DC algorithm are also developed for specific models. For example, \cite{niu2019high} boosts the projected DC method via a sum-of-square decomposition in the portfolio model of the first 4th-order moments. \cite{zhang2024boosted} accelerates the DCA with an exact line search to address linearly constrained polynomial programs based on a power-sum-DC decomposition.

The purpose of this paper is to propose an improved boosted difference of convex functions algorithm (IBDCA) to extrapolate DCA via monotone line search when only the smoothness of the second DC component is guaranteed. A key observation is that, in this type of DC problem, the BDCA search direction may lead to an ascent for the current iteration's point, which, however, is a descent for the point generated by the last DCA iteration. Based on this observation, the core idea of our proposed IBDCA is to perform a linear search in a descending direction from the point generated in the previous DCA iteration, instead of from the current iteration point like other BDCA-type methods. By setting appropriate parameters and search criteria, our algorithm abandons the additional search step of BDCA for points where the search direction may be ascending, instead adopting the step size of DCA. For points where the search direction is descending, a linear search and a step size similar to those of BDCA are used. To show the effectiveness of our method, we present an application of IBDCA to an important high-dimensional DC problem in image recovery: Cauchy noise removal.
%and the multiplicative Gamma noise removal.
Numerical experiments show that IBDCA outperforms DCA and other state-of-the-art algorithms in both restoration results and computation time. Furthermore, similar to BDCA, IBDCA has a greater possibility of escaping from poor local solutions than DCA and nmBDCA.

%%% contribution %%%
% To the best of our knowledge, our work is the first monotone BDCA-type method that can be used when $g$ is non-smooth. 
It is worth mentioning that BDCA adopts the search direction proposed earlier in \cite{fukushima1981generalized,mine1981minimization}, and that they also came up with the idea of conducting a line search from the point generated in the previous iteration. However, the Fukushima-Mine (F-M) method is very different from our proposed method. Actually, the maximum step size of the F-M method is equal to $1$ (the same as DCA), while IBDCA can accelerate DCA in many cases. We demonstrate that, in many situations, the linear search of our proposed method can achieve results similar to those of BDCA. On the other hand, in \cite{mine1981minimization}, the Fukushima-Mine method did not provide an accurate estimate of the descent of the objective function, nor did it specify concrete linear search criteria. In contrast, we prove that under our designed search criteria, our method exhibits similar convergence and descent properties to BDCA. Therefore, our proposed IBDCA extends BDCA to a broader class of non-smooth objective functions. We also give the global convergence and convergence rate of the IBDCA based on Kurdyka–Łojasiewicz (K-Ł) property, which covers a great number of models in application such as image processing.

% The outline is not required, but we show an example here.
The rest of this paper is organized as follows. In \cref{sec:pre}, we recall some preliminary concepts and results and give the assumptions throughout the paper. The proposed IBDCA and basic properties are presented in \cref{sec:MBDCA}. The convergence analysis and convergence rate under K-Ł property are given in \cref{sec:KL}. In \cref{sec:exp}, we provide applications of IBDCA in image recovery, where we compare the performance of IBDCA with state-of-the-art algorithms. Finally, some conclusions and future research are briefly discussed in \cref{sec:end}.

\section{Preliminaries}
\label{sec:pre}

In this section, we recall some preliminary concepts and results that will be used in the sequel. We also state our assumptions imposed on \eqref{DCproblem}. Throughout this paper, the inner product of two vectors $x, y \in \mathbb{R}^{m}$ is denoted by $\langle x, y\rangle$, while $\|\cdot\|$ denotes the induced norm. We use $\mathbb{B}(x, r)$ to denote the open ball centered at $x$ with radius $r>0$. The gradient of a function $f: \mathbb{R}^{m} \rightarrow \mathbb{R} \cup\{+\infty\}$ which is differentiable at some point $x \in \mathbb{R}^{m}$ is denoted by $\nabla f(x)$.

\subsection{Basic conceptions and results in convex analysis}
\label{subsec:opt}

In this subsection, we review some basic definitions and related conclusions regarding convex analysis and generalized derivatives of non-smooth functions.

\begin{definition}\label{def:dompro}
Let $f: \mathbb{R}^{m} \rightarrow \mathbb{R} \cup\{+\infty\}$ be an extended real-valued function, where $\mathbb{R}^{m} \rightarrow \mathbb{R} \cup\{+\infty\}$ denotes the extended real line. The domain of $f$ is defined as
$$\operatorname{dom} f=\left\{x \in \mathbb{R}^{m}: f(x)<+\infty\right\}.$$
$f$ is said to be proper if $\operatorname{dom} f\neq \emptyset$.
\end{definition}

\begin{definition}\label{def:convex}
  A function $f: \mathbb{R}^{m} \rightarrow \mathbb{R} \cup\{+\infty\}$ is said to be convex if
$$f(\lambda x+(1-\lambda) y) \leq \lambda f(x)+(1-\lambda) f(y) \quad \text { for all } x, y \in \mathbb{R}^{m} \text { and } \lambda \in[0,1].$$
$f$ is called strongly convex with modulus $\rho>0$ if $f-\frac{\rho}{2}\|\cdot\|^{2}$ is convex.
\end{definition}

\begin{definition}\label{def:lip}
  A function $F: \mathbb{R}^{m} \rightarrow \mathbb{R}^{n}$ is said to be Lipschitz continuous if there exists some constant $L\geq 0$ such that
$$\left \| F(x)-F(y) \right \|\leq L\left \| x-y \right \|, \quad \text { for all } x, y \in \mathbb{R}^{m}.$$
$F$ is said to be locally Lipschitz if, for every $x\in\mathbb{R}^{m}$, there exists a neighborhood $U_{x}$ of $x$ such that $F$ restricted to $U_{x}$ is Lipschitz continuous.
\end{definition}

% \begin{definition}\label{def:subdiff}
%   For a function $f: \mathbb{R}^{m} \rightarrow \mathbb{R} \cup\{+\infty\}$, its subdifferential $\partial f(\tilde{x})$ at $x\in \operatorname{dom}f$ is defined as
% $$\partial f(\tilde{x})=\left\{ y\in\mathbb{R}^{m}:f(x)-f(\tilde{x})\geq \left \langle y,x-\tilde{x} \right \rangle,\forall x\in \mathbb{R}^{m}.  \right\} $$
% For $\tilde{x}\notin \operatorname{dom}f$, we define $\partial f(\tilde{x})=\emptyset$. Especially, if $f$ is differentiable at $\tilde{x}$, then $\partial f(\tilde{x})=\left\{ \nabla f(\tilde{x})\right\}$. 

% $x\in \mathbb{R}^{m}$ is said to be a critical point of $f$ if $0\in\partial f(\tilde{x})$.
% \end{definition}
\begin{proposition}[\cite{Pallaschke2016AnEP}]\label{subdiff_monotone}
  Let $f: \mathbb{R}^{m} \rightarrow \mathbb{R} \cup\{+\infty\}$ be an extended real-valued function, then $f$ is convex if and only if
  \begin{displaymath}
    \left\langle u-v,x-y\right\rangle\geq 0,\quad \text { for all } x, y \in \mathbb{R}^{m} \text { and } u\in\partial f(x),v\in\partial f(y).
  \end{displaymath}
$f$ is $\rho$-strongly convex ($\rho>0$) if and only if
  \begin{displaymath}
    \left\langle u-v,x-y\right\rangle\geq\rho\left\|x-y\right\|^{2},\quad \text { for all } x, y \in \mathbb{R}^{m} \text { and } u\in\partial f(x),v\in\partial f(y).
  \end{displaymath}
  \label{th2.5}
\end{proposition}

\begin{definition}
Let $f:\mathbb{R}^{m}\to\mathbb{R}\cup\left\{ +\infty \right\}$ be locally Lipschitz continuous, the Clarke subdifferential $\partial_{C}f(\tilde{x})$ of $f$ at $\tilde{x}$ is defined as
$$
\partial_{C}f(\tilde{x})=\operatorname{co}\left\{\lim_{x \to \tilde{x},x\in D_{f}}\nabla f(x)  \right\},
$$
where $\operatorname{co}$ represents the convex hull, $D_f$ is the set of non-differentiable points of $f$, which is of zero Lebesgue measure due to the Rademacher theorem \cite{evans2018measure}.
\end{definition}

\begin{proposition}[Property of Clarke subdifferential, \cite{clarke1990optimization}]

(1) If $f$ is convex in a neighborhood of $x$, then
$$
\partial_{C}f(x)=\partial f(x).
$$

(2) If $f$ is differentiable at $x$, then we have
$$
\partial_{C}f(x)=\left\{ \nabla f(x) \right\}.
$$

(3) For any constant $s\in\mathbb{R}$, we have
$$
\partial_{C}(sf)(x)=s\partial_{C}f(x).
$$

(4) $$\partial_{C}(f+g)(x)\subseteq \partial_{C}f(x)+\partial_{C}g(x).$$ In addition, the equality holds if either $f$ or $g$ is differentiable at $x$.
\end{proposition}

\begin{definition}\label{def:directional_dev}
Let $f:\mathbb{R}^{m}\to\mathbb{R}\cup\left\{ +\infty \right\}$ be an extended real-valued function. For $x\in\operatorname{dom}f$ and $d\in\mathbb{R}^{m}$, the directional derivative of $f$ at $x$ in the direction $d$ is defined as:
$$
f^{\prime}(x ; d):=\lim _{t \downarrow 0} \frac{f(x+t d)-f(x)}{t}.
$$
\end{definition}

\subsection{Assumptions}
Throughout this paper, we impose the following three hypotheses on $g$ and $h$ in \eqref{DCproblem}.
\begin{assumption}\label{hypo1}
 Both $g$ and $h$ are strongly convex with modulus $\rho>0$;
\end{assumption}

\begin{assumption}\label{hypo2}
For any $x\in\operatorname{dom}g$ and any $d\in\mathbb{R}^{m}$,
\begin{equation}
{g}'(x;d)=\sup_{v\in\partial g(x)}\left\langle v,d\right\rangle; \label{hypo:eq2}
\end{equation}
\end{assumption}

\begin{assumption}\label{hypo3}
The function $h$ is continuously differentiable with locally Lipschitz gradient on an open set containing $\operatorname{dom}g$ and
\begin{equation}
\inf_{x\in\mathbb{R}^{m}}\phi(x)>-\infty.\label{hypo:eq3}
\end{equation}
\end{assumption}

\begin{remark}
\cref{hypo1} is not restrictive since any DC decomposition of $\phi$ as $\phi=g-h$ can be expressed as $\phi=\left ( g+\frac{\rho}{2}\left\| \cdot \right\|^2\right )-\left ( h+\frac{\rho}{2}\left\| \cdot \right\|^2\right )$. The equality \eqref{hypo:eq2} holds for all $x\in\operatorname{ri dom}g$ (see \cite{rockafellar1970convex}). Especially, if $\operatorname{dom}g=\mathbb{R}^m$, then \cref{hypo2} is clearly valid. \eqref{hypo:eq3}
is common in the context of DC programming, see, e.g., \cite{nonsmoothBDCA,aragon2018accelerating,Cruz2020generalized}.
\end{remark}

Under these hypotheses, the following necessary optimality condition holds.

\begin{proposition}[First-order necessary optimality condition,  \cite{fukushima1981generalized}]
If $x^{\ast}\in\operatorname{dom}\phi$ is a local minimizer of \cref{DCproblem}, then $x^{\ast}$ is a critical point, i.e.,
$$
\nabla h(x^{\ast})\in\partial g(x^{\ast}).
$$
\end{proposition}

\section{Monotone Boosted DC Algorithms}
\label{sec:MBDCA}

In this section, we present the IBDCA to solve \cref{DCproblem}, and further provide the convergence of this proposed algorithm.

\subsection{Improved BDCA}

We first present the conceptual statement of DCA and BDCA. The main idea of the DCA is to use the affine majorization of $h$ in the problem \cref{DCproblem} to approximate $f(x)$, and then minimize the resulting convex function. The main steps of the DCA are shown in \cref{alg:DCA}.

% \begin{algorithm}
% \caption{Calculate $y = x^n$}\label{algo1}
% \begin{algorithmic}[1]
% \Require $n \geq 0 \vee x \neq 0$
% \Ensure $y = x^n$ 
% \State $y \Leftarrow 1$
% \If{$n < 0$}\label{algln2}
%         \State $X \Leftarrow 1 / x$
%         \State $N \Leftarrow -n$
% \Else
%         \State $X \Leftarrow x$
%         \State $N \Leftarrow n$
% \EndIf
% \While{$N \neq 0$}
%         \If{$N$ is even}
%             \State $X \Leftarrow X \times X$
%             \State $N \Leftarrow N / 2$
%         \Else[$N$ is odd]
%             \State $y \Leftarrow y \times X$
%             \State $N \Leftarrow N - 1$
%         \EndIf
% \EndWhile
% \end{algorithmic}
% \end{algorithm}

\begin{algorithm}
\caption{DCA ((DC Algorithm, \cite{AN1996DCA})}
\label{alg:DCA}
\begin{algorithmic}[1]
\State  Let $x^{0}$ be any initial point, select any $y^{0}\neq x^{0}$ and set $k \coloneqq 0$.
\State  Solve the strongly convex minimization problem to obtain the unique solution $y^{k}$:
$$
\min_{x \in \mathbb{R}^m} g(x)-\left\langle\nabla h\left(x^k\right), x\right\rangle
$$

\State  If $y^{k} = x^{k}$, then stop and return $x^{k}$. Otherwise, we set $x^{(k+1)} \coloneqq y^k$, set $k \coloneqq k + 1$, and go to the previous step.
\end{algorithmic}
\end{algorithm}
The main idea of BDCA is to extend the DCA by performing a line search along the direction $d^{k}$ obtained from each DCA iteration. The main steps of BDCA are shown in \cref{alg:BDCA}.
\begin{algorithm}
\caption{BDCA (Boosted DC algorithm, \cite{aragon2018accelerating})}
\label{alg:BDCA}
\begin{algorithmic}[1]
\State{Fix $\alpha > 0$ and $0 < \beta < 1$. Let $x^{0}$ be any initial point, and set $k \coloneqq 0$.}

\State{Solve the strongly convex minimization problem to obtain the unique solution $y^{k}$:
$$
\min_{x \in \mathbb{R}^m} g(x)-\left\langle\nabla h\left(x^k\right), x\right\rangle
$$
}
\State{Set $d^{k} \coloneqq y^{k} - x^{k}$. If $d^{k} = 0$, stop and return $x^{k}$. Otherwise, go to the next step.}
\State{Take $\overline{\lambda}_{k}>0$ and set $\lambda_{k} \coloneqq \overline{\lambda}_{k}$. When $\phi(y^{k} + \lambda_{k}d^{k}) > \phi(y^{k})  -\alpha \lambda_k\left\|d^k\right\|^2$
, let $\lambda_{k} \coloneqq \beta \lambda_{k}$.}
\State{Set $x^{k+1} \coloneqq y^{k} + \lambda_{k} d^{k}$. If $x^{k+1} \coloneqq x^{k}$, stop and return $x^{k}$; otherwise, set $k \coloneqq k + 1$ and go to step 2.}

\end{algorithmic}
\end{algorithm}

\cite{nonsmoothBDCA} points out that if the first function in the DC decomposition is non-smooth, the line search in BDCA may not be applicable. The idea of nmBDCA allows the line search to be non-monotone, where the objective function value may increase during the line search process. The main steps of nmBDCA are shown in \cref{alg:nmBDCA}.
\begin{algorithm}
\caption{nmBDCA (Nonmonotone boosted DC algorithm, \cite{ferreira2024boosted})}
\label{alg:nmBDCA}
\begin{algorithmic}[1]
\State{Fix $\alpha > 0$ and $0 < \beta < 1$. Let $x^{0}$ be any initial point, and set $k \coloneqq 0$.}

\State{Solve the strongly convex minimization problem to obtain the unique solution $y^{k}$:
$$
\min_{x \in \mathbb{R}^m} g(x)-\left\langle\nabla h\left(x^k\right), x\right\rangle
$$
}
\State{Set $d^{k} \coloneqq y^{k} - x^{k}$. If $d^{k} = 0$, stop and return $x^{k}$. Otherwise, go to the next step.}
\State{Take $\overline{\lambda}_{k}>0$, set $\lambda_{k} \coloneqq \overline{\lambda}_{k}$, and choose $v_{k}>0$. When $\phi(y^{k} + \lambda_{k}d^{k}) > \phi(y^{k})  -\alpha \lambda_k\left\|d^k\right\|^2 + v_{k}$
, let $\lambda_{k} \coloneqq \beta \lambda_{k}$.}
\State{Set $x^{k+1} \coloneqq y^{k} + \lambda_{k} d^{k}$. If $x^{k+1} \coloneqq x^{k}$, stop and return $x^{k}$; otherwise, set $k \coloneqq k + 1$ and go to step 2.}

\end{algorithmic}
\end{algorithm}
However, when the direction of $d^{k}$ is ascending, the line search of nmBDCA cannot reduce the objective function value. We thus propose a monotonically decreasing improved BDCA to solve such DC problems.

\begin{algorithm}
\caption{IBDCA(Improved boosted DC algorithm)}
\label{alg:modified_BDCA}
\begin{algorithmic}[1]
\State{Fix $\overline{\lambda} > 1$ and $0 < \beta < 1$. Let $x^{0}$ be any initial point, and set $k \coloneqq 0$.}

\State{Solve the strongly convex minimization problem to obtain the unique solution $y^{k}$:
\begin{equation}
\min_{x \in \mathbb{R}^m} g(x)-\left\langle\nabla h\left(x^k\right), x\right\rangle
\label{strongconvex}
\end{equation}

}
\State{Set $d^{k} \coloneqq y^{k} - x^{k}$. If $d^{k} = 0$, stop and return $x^{k}$. Otherwise, go to the next step.}
\State{Take $\lambda_k>0$ and choose $\alpha_{k}>0$. Process $\lambda_{k} \coloneqq \beta\lambda_{k}$ until $\lambda^{k}$ satisfies: 
\begin{equation}
\phi(x^{k} + \lambda_{k}d^{k}) \leq \phi(x^{k})  -\alpha \lambda_k\left\|d^k\right\|^2,
\label{eq5}
\end{equation}
and
\begin{equation}
\phi(x^{k} + \lambda_{k}d^{k}) \leq \phi(y^{k}).
\label{proofcritical1}
\end{equation}
Set $\lambda^{k}=1$ and go to the next step, if $\lambda^{k} \leq 1$. }
\State{Set $x^{k+1} \coloneqq x^{k} + \lambda_{k} d^{k}$. If $x^{k+1} \coloneqq x^{k}$, stop and return $x^{k}$; otherwise, set $k \coloneqq k + 1$ and go to step 2.}

\end{algorithmic}
\end{algorithm}

Our IBDCA uses the same linear search direction as BDCA. However, the main difference is that the linear search of BDCA starts from $y^{k}$, while IBDCA chooses $x^{k}$ as the starting point as the line search. When $\lambda^{k} =1 $ in step 4 of \cref{alg:modified_BDCA}, this iteration is equivalent to performing a step of DCA without the additional descent step of BDCA. This allows us to avoid the situation where BDCA cannot be applied when the search direction is ascending. Furthermore, we will point out later that our proposed IBDCA is a generalization of BDCA, and in many cases, it could conduct a similar additional descent step as BDCA.

\subsection{Basic Properties of IBDCA}

Our IBDCA, similar to DCA, has the following properties:

\begin{proposition}\label{def:algo4}
For any $k \in \mathbb{N}$, the following properties hold:
\begin{enumerate}[label=\arabic*.]
    \item If \(d^{k} = 0\), then \(x^{k}\) is a critical point of \(\phi\).
    \item \(\phi(y^{k}) \leq \phi(x^{k}) - \rho \left\|d^k\right\|^2\) holds.
\end{enumerate}
\label{prop3.1}
\end{proposition}
\begin{proof}
For the first assertion, since $y^{k}$ is the unique solution of the strongly convex problem \cref{strongconvex}, we have:
\begin{equation}
\nabla h(x^{k}) \in \partial g(y^{k}).
\label{equ7}
\end{equation}
If $d^{k} = y^{k} - x^{k} = 0$, then
\begin{equation*}
0 \in \partial g(x^{k}) - \nabla h(y^{k}) = \partial \phi(x^{k}).
\label{prop3.1.7}
\end{equation*}

For the second assertion, using the $\rho$-strong convexity of $g$ and $h$, we obtain:
\begin{equation*}
g(x^{k}) - g(y^{k}) \geq -\langle v, d^{k}\rangle + \rho/2 \left\|d^k\right\|^2, \forall v \in \partial g(y^{k}),
\label{firstcritical}
\end{equation*}
\begin{equation*}
h(y^{k}) - h(x^{k}) \geq \langle\nabla h(x^{k}), d^{k}\rangle + \rho/2 \left\|d^k\right\|^2.
\end{equation*}
Then the result follows by summing these two inequalities and applying \cref{equ7}.
\end{proof}

The following result shows that $d^{k}$ is a descent direction of $\phi$ at $x^{k}$, thus providing a reasonable line search strategy for IBDCA. 

\begin{proposition}\label{def:descent_direction}
For any $k \in \mathcal{N}$, the following properties hold:\\
\begin{enumerate}
\item $\quad \phi^{\prime}\left(x^k;d\right) \leq -\rho\left\|d^k\right\|^2;$\\
\item $\quad
\forall \, 0 \leq \alpha<\rho, \exists \, \sigma_{k} >0$, \cref{eq5} holds for any  $\left.\left.\lambda_{k} \in \right] 0, \sigma_k\right].$
\end{enumerate}
\label{prop3.2}
\end{proposition}
\begin{proof}
By the definition of the one-sided directional derivative and \cref{hypo2}, we have:
\begin{equation}
\begin{aligned}
\phi^{\prime}\left(x^k ; d^k\right) & =\lim _{t \downarrow 0} \frac{\phi\left(x^k+t d^k\right)-\phi\left(x^k\right)}{t} \\
& =\lim _{t \downarrow 0} \frac{g\left(x^k+t d^k\right)-g\left(x^k\right)}{t}-\lim _{t \downarrow 0} \frac{h\left(x^k+t d^k\right)-h\left(x^k\right)}{t} \\
& =\sup \left\{\left\langle v, d^k\right\rangle, v \in \partial g\left(x^k\right)\right\}-\left\langle\nabla h\left(x^k\right), d^k\right\rangle \\
& =\sup \left\{\left\langle v-\nabla h\left(x^k\right), d^k\right\rangle, v \in \partial g\left(x^k\right)\right\} .
\label{descentproof}
\end{aligned}
\end{equation}
According to the \cref{subdiff_monotone} and \cref{equ7}, 
\begin{equation}
\left\langle\nabla h\left(x^k\right)-v, d^k\right\rangle \geq \rho\left\|d^k\right\|^2, \forall v \in \partial g\left(x^k\right)
\label{proofeq8}
\end{equation}
Then combining \cref{descentproof} with \cref{proofeq8} gives the first assertion. 

For the second assertion, if $d^{k} = 0$, then the proof is complete. Otherwise, according to \cref{descentproof} and the first assertion, we have:
\begin{equation*}
\lim _{t \downarrow 0} \frac{\phi\left(x^k+t d^k\right)-\phi\left(x^k\right)}{t} \leq-\rho\left\|d^k\right\|^2<-\alpha\left\|d^k\right\|^2.
\end{equation*}
Hence, there exists some $\sigma_{k} \geq 0$, such that:
\begin{equation*}
\frac{\phi\left(x^k+\lambda d^k\right)-\phi\left(x^k\right)}{\lambda} \leq-\alpha\left\|d^k\right\|^2, \forall \lambda \in \left] 0, \sigma_k \right].
\end{equation*}
\end{proof}

Our IBDCA can be viewed as a generalization of the classical BDCA, as shown in the following result.

\begin{proposition}
For any $k \in \mathcal{N}$, the following properties hold:\\
\begin{enumerate}
\item Take $\alpha_{k} \leq \rho$, if $\phi'(y^{k};d^{k}) \leq -\alpha_{k}\left\|d^k\right\|^2$, then there exists $\tau_{k} \geq 0$, such that for any  $\lambda_k \in \left]1,1+\tau_k\right]$, \cref{eq5} and \cref{proofcritical1} are satisfied. 

\item  Take $\alpha_{k} \leq \rho$, if $\phi$ is smooth at point $y^{k}$, then there exists $\tau_{k} \geq 0$, such that for any  $\lambda_k \in \left]1,1+\tau_k\right]$, \cref{eq5} and \cref{proofcritical1} are satisfied. 

% \item if $\sup \left\{\|u-v\|, u, v \in \partial g\left(x^k\right)\right\} \leq \epsilon_k$, and $\rho-\frac{\epsilon_k}{\left\|d^k\right\|}>0$. If we take $\alpha_k<\rho-\frac{\epsilon_k}{\left\|d^k\right\|}$, then there exists $\tau_k>0$, such that for any $\lambda_k \in\left] 1,1+\tau_k\right]$, \cref{eq5} and \cref{proofcritical1} are satisified. 
\end{enumerate}
\label{prop3.3}
\end{proposition}

% \begin{remark}
% (1) \cref{prop3.3} implies that our modified BDCA can be viewed as a generalization of the BDCA;\\
% (2) Regarding the condition in assertion (3) of \cref{prop3.3}, in practical applications, the function $g$ can often be written as $g(x) = r(x) + \bar{g}(x)$, where r(x) is a possibly non-smooth regularization term of the form of $\mu \|A(\cdot)\|$, where $\mu > 0$ denotes the weight, and A is a bounded linear operator. In real applications, $\mu$ is usually quite small, correspondingly, $\sigma_{k}$ is also relatively small. Moreover, in the initial stage of the iteration, $\|d^{k}\|$ is also relatively large, so the condition in (3) of \cref{prop3.3} can potentially be satisfied.
% \end{remark}
\begin{proof} 
 According to the definition of the directional derivative, we have:
$$
\lim _{t \downarrow 0} \frac{\phi\left(y^k+t d^k\right)-\phi\left(y^k\right)}{t}<-\alpha_k\left\|d^k\right\|^2,
$$
Thus, there exists $\tau_{k} \geq 0$ such that for any $\lambda \in\left] 1,1+\tau_k\right]$, we have:
$$
\phi\left(y^k+(\lambda-1) d^k\right)-\phi\left(y^k\right) \leq-\alpha_k(\lambda-1)\left\|d^k\right\|^2,
$$
further, we have:
$$
\phi\left(x^k+\lambda d^k\right) \leq \phi\left(y^k\right)-\alpha_k(\lambda-1)\left\|d^k\right\|^2 \leq \phi\left(x^k\right)-\lambda \alpha_k\left\|d^k\right\|^2 .
$$
The last inequality comes from \cref{def:algo4} and $\alpha_{k} \leq \rho$.

According to the \cref{equ7} in \cref{prop3.1} and $\alpha_{k} < \rho$, we have:
$$
\phi^{\prime}\left(y^k ; d^k\right)=\left\langle\nabla \phi\left(y^k\right), d^k\right\rangle \leq-\rho\left\|d^k\right\|^2<-\alpha_k\left\|d^k\right\|^2,
$$
The remaining part can be proved using the same method as in the first assertion.
\end{proof}
% According to \cref{hypo2}, we have:
% \begin{equation}
% \begin{aligned}
% \phi^{\prime}\left(y^k ; d^k\right) & =\sup \left\{\left\langle v, d^k\right\rangle, v \in \partial g\left(y^k\right)\right\}-\left\langle\nabla h\left(y^k\right), d^k\right\rangle \\
% & =\sup \left\{\left\langle v-\nabla h\left(x^k\right), d^k\right\rangle, v \in \partial g\left(y^k\right)\right\}\\
% &-\left\langle\nabla h\left(x^k\right)-\nabla h\left(y^k\right), d^k\right\rangle \\
% & \leq \sup \left\{\left\|v-\nabla h\left(x^k\right)\right\|\left\|d^k\right\|, v \in \partial g\left(y^k\right)\right\}-\rho\left\|d^k\right\|^2 \\
% & \leq \epsilon_k\left\|d^k\right\|-\rho\left\|d^k\right\|^2 \\
% & <-\alpha_k\left\|d^k\right\|^2
% \end{aligned}
% \end{equation}
% The first inequality is obtained by applying the Cauchy inequality and \cref{th2.5}, and the second inequality employs \cref{equ7}. The remaining part can be proved using the same method as assertion (1).

Next, we present the convergence results of the sequence generated by IBDCA.
\begin{theorem}
For any initial point $x^{0} \in \mathbb{R}^{m}$, IBDCA either returns a critical point or generates an infinite sequence that satisfies the following properties:
\begin{enumerate}
\item The sequence $\left\{\phi(x^{k})\right\}$ is monotonically decreasing and converges to $\phi^{*}$;

\item Any limit point of ${x^{k}}$ is a critical point of the DC problem \cref{DCproblem}. Furthermore, if $\phi$ is coercive, i.e., $\phi(x) \rightarrow+\infty$ as $\|x\| \rightarrow \infty$, then there exists a subsequence of $x^{k}$ that converges to a critical point of the DC problem \cref{DCproblem};

\item $\sum_{k=0}^{\infty}\left\|d^k\right\|^2<\infty$, and consequently, $ \sum_{k=0}^{\infty}\left\|x^{k+1}-x^k\right\|^2<\infty
$.

\end{enumerate}
\label{prop3.7}
\end{theorem}

\begin{proof} If the algorithm returns $x^{k}$, according to \cref{prop3.1}, ${x^{k}}$ is a critical point. Otherwise, we have:
\begin{equation}
\phi(x^{k+1}) \leq \phi(y^k) \leq  \phi(x^k) - \rho\|d^k\|^2.
\label{equ10}
\end{equation}
Since the sequence is monotonically decreasing and has a lower bound according to \cref{hypo3}, it will converge to ${\phi^{*}}$.

Therefore, we have
\begin{equation*}
\phi(x^k)-\phi(x^{k+1}) \rightarrow 0,
\end{equation*}
which yields $\|d^{k}\| \rightarrow 0$ from \cref{equ10}.

Let $x^{\ast}$ be any critical point of ${{x}^{k}}$, and suppose $\left\{{x}^{k_{n}}\right\}$ is a subsequence that converges to ${x}^{*}$. Since $\|{y}^{k_n} - {x}^{k_n}\|^2 = \|{d}^{k_n}\|^2 \to 0$, we have ${y}^{k_n} \to {x}^{*}$. By continuity, we obtain: 
$$
\nabla h({y}^{k_n}) \to \nabla h({x}^{*}).
$$

Moreover, $\nabla h({y}^{k_n}) \in \partial g({x}^{k_n})$, and the graph of $\partial g$ is closed. Therefore, we have 
$$
\nabla h(x^*) \in \partial g({x}^*).
$$

Especially, when $\phi$ is coercive, the first assertion implies that the sequence $\{{x}^k\}$ is bounded. Hence, there exists some subsequence that converges to a critical point of the DC problem \eqref{DCproblem}.

For the last assertion, since we have:
$$
\rho\left\|d^k\right\|^2 \leq \phi\left(x^k\right)-\phi\left(x^{k+1}\right).
$$
Then, summing this inequality from $k=0$ to $N$, we obtain the following:
$$
\sum_{k=0}^N \rho\left\|d^k\right\|^2 \leq \phi\left(x^0\right)-\phi\left(x^{N+1}\right) \leq \phi\left(x^0\right)-\inf _{x \in \mathbb{R}^m} \phi(x),
$$
which completes the proof by letting $N \to \infty$.
\end{proof}

Finally, we present two examples to show the strength of IBDCA. The first example is given in \cite[Example 3.2]{nonsmoothBDCA}, which reveals that when the function $g$ in the DC decomposition is non-smooth, even if $y^{k}$ is not a global minimizer, $d^{k}$ may still be an ascending direction of $y^{k}$. In this example, the BDCA fails to complete the task, nmBDCA is inefficient since it increases the function value in the ascending direction, while the improved BDCA can solve this problem effectively by monotone line search.

\begin{exmp}
We consider the following DC problem:
$$
\min_{x=(u, v) \in \mathbb{R}^2} g(u, v)-h(u, v):=\phi(x),
$$
where 
$$
g(u, v)=-\frac{5}{2} u+u^2+v^2+|u|+|v|, \quad h(u, v)=\frac{1}{2}\left(u^2+v^2\right) \text {, }
$$
Set $x^{0} = (u^{0}, v^{0})=(\frac{1}{2}, 1).$ Then, the point generated by DCA is $y^{0} = (1,0)$. As stated in \cite[Example 3.2]{nonsmoothBDCA}, since $d^{0}=y^{0} - x^{0} = (\frac{1}{2}, -1)$ is not a descending direction at point $y^{0}$, the line search of BDCA cannot be performed.

However, our improved BDCA abandons the additional search at $y^{0}$. We let $x^{1}=y^{0}$, and thus obtain $y^{1} = (\frac{5}{4}, 0)$ using DCA. Subsequently, $x^{2}=(\frac{3}{2},0)$ is a point that satisfies the search criteria of IBDCA. In fact, with $d^{1} = y^{1} - x^{1} = (\frac{1}{4},0)$, we then obtain the following with simple calculation:
$$
\phi\left(x^2\right)=\phi\left(x^1+2 d^1\right)<\phi\left(x^1\right)-2 \alpha\left\|d^1\right\|^2, \quad \forall \alpha \in \left] 0,2 \right[.
$$
Moreover, $x^{1}$ is not a minimizer of $\phi$, the BDCA thus fails to solve the model. The points generated by DCA converge to the global minimum, while the convergence speed of the improved BDCA is significantly accelerated, as shown in \cref{Fig_Example1}.
\label{eg3.5}
\end{exmp}

\begin{figure*}[!ht]
\centering
\begin{minipage}[t]{0.39\linewidth}
\centering
\includegraphics[width=0.99\linewidth]{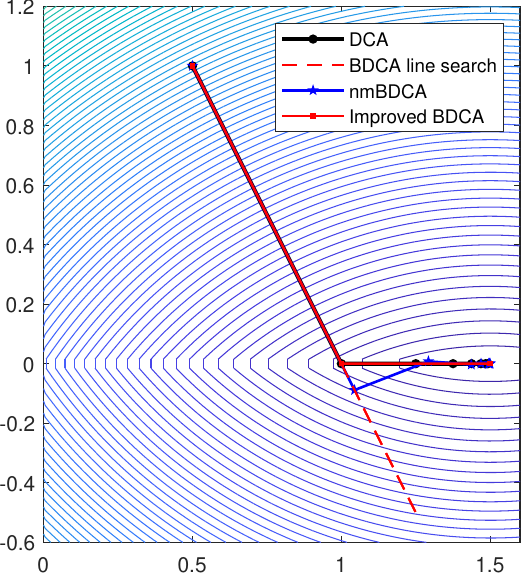}
\end{minipage}%
\vspace{0.1cm}
\begin{minipage}[t]{0.59\linewidth}
\centering
\includegraphics[width=0.99\linewidth]{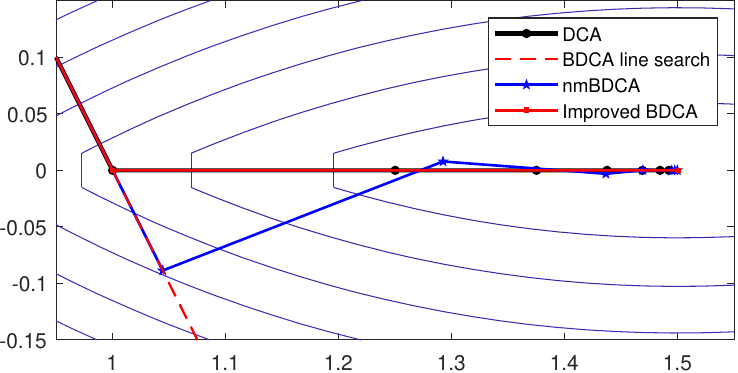}
\end{minipage}
\caption{The iterative behavior of DCA, nmBDCA, BDCA and IBDCA in \cref{eg3.5}.}
\centering
\label{Fig_Example1}
\end{figure*}

\cite[Example 3.1]{nonsmoothBDCA} points out that the BDCA possesses a greater possibility of escaping from non-global solutions. The following example, adapted from the SCAD (smoothly clipped absolute deviation, \cite{fan2001variable}) function, shows that our improved BDCA also has this advantage.

\begin{exmp}
We consider the following DC problem:
$$
\min_{x=(u, v) \in \mathbb{R}^2} \phi(x):=\tilde{\phi}(u)+\tilde{\phi}(v),
$$
where

\begin{equation*}
\begin{split}
\tilde{\phi}(u)&=\tilde{g}(u)-\tilde{h}(u),\\
\tilde{g}(u)&=\left\{\begin{array}{ll}
|x|+\frac{x^2}{5} & , \text { if }|x|<2 \\
|x|+(|x|-2)^2+\frac{x^2}{5} & , \text { if }|x| \geq 2
\end{array},\right. \\
\tilde{h}(u)&= \begin{cases}\frac{x^2}{5} & ,|x| \leq 1 \\
\frac{x^2-2|x|+1}{2}+\frac{x^2}{5} & , 1<|x|<2 . \\
\operatorname{sign}(x)+\frac{x^2}{5} & ,|x|>2\end{cases}
\end{split}
\end{equation*}
This DC problem satisfies \cref{hypo1,hypo2,hypo3}. 

\cref{Fig_Example2} shows the iterative behavior of DCA, nmBDCA, and improved BDCA with the initial point $x^{0} = (2.2, 0.4)$. In the improved BDCA, we take $\alpha_{k}=0.2$, $\overline{\lambda_{k}}=3$, $\beta_{k}=0.7$, and for nmBDCA, we take $\alpha_{k}=0.2$, 
 $\beta_{k}=0.7$, $\overline{\lambda_{k}}=3-1$, $\nu_{k}=\left\|d^k\right\|^2 /(k+1)$. Our improved BDCA not only has a faster convergence speed than DCA and nmBDCA, but the line search process also prevents the generated sequence from converging to the critical local solution $(2, 0)$, which is not even a local minimum point.

We randomly select one million initial points in the range $[0,3]\times[0,3]$. Numerical experiments show that the sequence generated by DCA is very likely to converge to the critical points of non-global minimum values, and nmBDCA still holds a portion of points that converge to the critical points of non-global minimum values. In contrast, the improved BDCA converges to the global minimum point for all initial points selected, and the convergence speed is faster than nmBDCA. 
$$
\begin{array}{|c|c|c|c|c|c|}
\hline & (2,2) & (0,2) & (2,0) & (0,0) & \text { Time } \\
\hline \text { DCA } & 110,831 & 222,087 & 223,147 & 443,935 & 13.10 \mathrm{~s} \\
\hline \text { nmBDCA } & 0 & 9645 & 9563 & 980792 & 276.7 \mathrm{~s} \\
\hline \text { improved BDCA } & 0 & 0 & 0 & 1,000,000 & 34.59 \mathrm{~s} \\
\hline
\end{array}
$$

\label{eg3.6}
\end{exmp}

\begin{figure*}[!ht]
\centering
\begin{minipage}[t]{0.59\linewidth}
\centering
\includegraphics[width=0.99\linewidth]{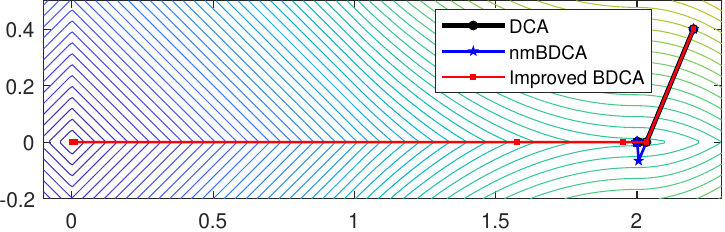}
\end{minipage}%
\vspace{0.1cm}
\begin{minipage}[t]{0.39\linewidth}
\centering
\includegraphics[width=0.99\linewidth]{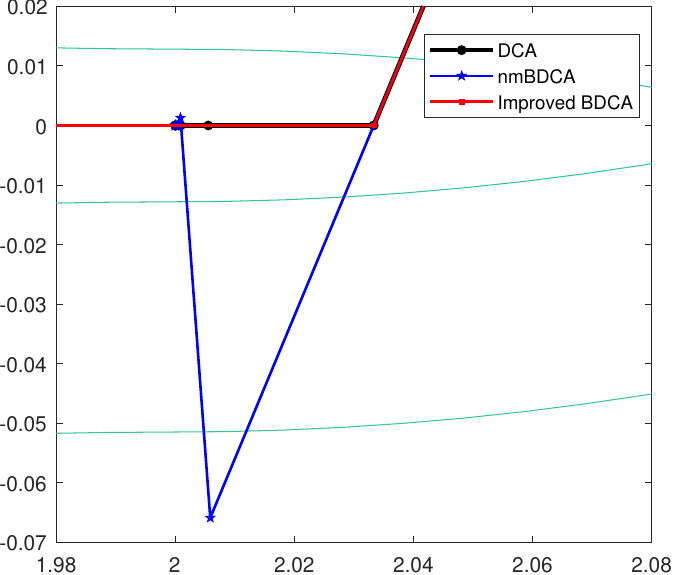}
\end{minipage}
\caption{The iterative behavior of DCA, nmBDCA, and improved BDCA with the initial point $x^{0} = (2.2, 0.4)$ in \cref{eg3.6}.}
\centering
\label{Fig_Example2}
\end{figure*}

\section{Global Convergence Properties Based on the Kurdyka–Łojasiewicz Inequality}
\label{sec:KL}
In this section, we establish the global convergence and convergence rate of the sequence generated by the improved BDCA based on the Kurdyka–
Łojasiewicz (K-Ł) inequality. First, we review the relevant concepts of the K-Ł inequality, which has been widely applied in DC (Difference of Convex) problems non-smooth optimization \cite{aragon2018accelerating,nonsmoothBDCA,le2009convergence, chouzenoux2016block,CHAI2022neurodynamic}.
\begin{definition}
Let $f: \mathbb{R}^m \to \mathbb{R}$ be a locally Lipschitz continuous function. We say that $f$ satisfies the K-Ł inequality at $x^* \in \mathbb{R}^m$ if there exist $\eta \in ]0, +\infty[$, a neighborhood $U$ of $x^*$, and a concave function $\psi: [0, \eta] \to [0, +\infty[$ such that:
\begin{enumerate}
\item $\psi(0) = 0$;
\item $\psi$ is of class $C^1$ on the interval $[0, \eta]$;
\item $\psi'[0, \eta] > 0$;
\item For any $x \in U$ satisfying $f(x^*) \leq f(x) \leq f(x^*) + \eta$, there exists:
\begin{equation}
\psi^{\prime}\left(f(x)-f\left(x^*\right)\right) \operatorname{dist}\left(0, \partial_C f(x)\right) \geq 1
\label{inequ11}
\end{equation}
\end{enumerate}
\end{definition}

Then, we provide the definitions of subanalytic and semialgebraic functions (referencing \cite{attouch2009convergence,attouch2010proximal,attouch2013convergence}), encompassing a broad class of functions commonly encountered in optimization problems. Under certain conditions, these functions satisfy the K-Ł inequality.

\begin{definition}
We say $A \subset \mathbb{R}^m$ is subanalytic if for any point $a \in A$, there exists a neighborhood $V$ of $a$, and some $d \geq 1$, such that $A \cap V = \{x \in \mathbb{R}^m: \exists y \in \mathbb{R}^d, (x, y) \in B\}$, where $B$ is bounded and satisfies:
$$
B=\bigcup_{i=1}^p \bigcap_{j=1}^q\left\{(x, y) \in \mathbb{R}^m \times \mathbb{R}^d: f_{i j}(x, y)=0, g_{i j}(x, y)<0\right\},
$$
for some real analytic functions $f_{ij}, g_{ij}: \mathbb{R}^m \times \mathbb{R}^d \to \mathbb{R} (1 \leq i \leq p, 1 \leq j \leq q)$. 

A function $f: \mathbb{R}^m \to \mathbb{R} \cup {+\infty}$ is said to be subanalytic, if its graph $\operatorname{gph}f = \{(x, t): x \in \mathbb{R}^m, f(x) = t\}$ is a subanalytic set of $\mathbb{R}^m \times \mathbb{R}$.
\end{definition}

\begin{definition}
We say a set $A \subset \mathbb{R}^m$ is semialgebraic, if A can be formed as
$$
A=\bigcup_{i=1}^p \bigcap_{j=1}^q\left\{x \in V: f_{i j}(x)=0, g_{i j}(x)<0\right\},
$$
where $f_{ij}, g_{ij}: \mathbb{R}^m \to \mathbb{R} (1 \leq i \leq p, 1 \leq j \leq q)$ are real polynomial functions.
\end{definition}

Below are some examples of common subanalytic functions in optimization problems (referencing \cite{attouch2009convergence,attouch2010proximal,attouch2013convergence}).

\begin{proposition}
The following conclusions hold:
\begin{enumerate}
\item The following functions are subanalytic:\\
(a) Real-analytic function; (b) Semialgebraic function; (c) The sum of a real-analytic function and a subanalytic function.
\item The following functions are semialgebraic:\\  
(a) The finite sum and product of semialgebraic functions;
(b) The composition of semi-algebraic functions;
(c) $l_p$ norm, i.e., $\|x\|_p := (\sum_i |x^i|^p)^{1/p}$, $p\in\mathbb{N}^+$.
    \end{enumerate}   
    \label{exmp4.4}
\end{proposition}

The following theorem plays a crucial role in the analysis of global convergence.

\begin{theorem} (\cite[Theorem 3.1]{bolte2007lojasiewicz}): Let $f: \mathbb{R}^m \to \mathbb{R} \cup \{+\infty\}$ be a subanalytic function, dom$f$ be closed, and $f|_{Domf}$ be continuous. Let $x^{\ast}$ be a critical point of $f$. Then $f$ satisfies the Kurdyka-Łojasiewicz inequality \cref{inequ11} at $x^{\ast}$. Besides, the concave function $\phi$ as in the Kurdyka-Łojasiewicz inequality can be chosen as $\phi(t) = Mt^{1-\theta}$ for some $M > 0$ and $\theta \in [0,1[$, where we use the convention $0^{0} =1$. $\theta$ is called the Łojasiewicz exponent. 
\label{thm4.5}
\end{theorem}

Next, we apply the above results to establish the global convergence theorem for the sequence generated by IBDCA.

\begin{theorem}
Let the DC function $\phi = g - h$ in \cref{{DCproblem}} be subanalytic, $\mathrm{dom}\phi$ be closed, and $\phi|_{Domf}$ be continuous. Suppose that $\phi$ is coercive and that the iterative sequence $\{x^k\}$ generated by the IBDCA has a subsequence converging to a critical point $x^*$ (by \cref{{prop3.7}}). Assume for any $k \in \mathbb{N}^+$, we have $\lambda_k \leq \bar{\lambda}$, for some $\bar{\lambda} > 1$, then the sequence $\left\{ x^k\right\}$ converges to $x^*$. 
\label{thm4.6}
\end{theorem}

\begin{proof}
Let $X$ be the set of cluster points of $\left\{x^k\right\}$. Then from \cref{prop3.7}, we have $\|d^k\| \to 0$, $\|x^{k+1} - x^k\| \to 0$ as $k \to \infty$. Note that for any $x \in X$, $\phi(x) \equiv \phi^* = \phi(x^*).$ Since $\|d^k\| = \|y^k-x^k\| \to 0$, then $X$ is also the set of all the cluster points of sequence $\{y^k\}$. Also, from
$$
\left\|y^k-y^{k+1}\right\| \leq\left\|x^{k+1}-y^k\right\|+\left\|d^{k+1}\right\| \\
\leq \|x^{k+1}-x^k\| + \|d^{k+1}\|,
$$
we have $\|x^{k+1} - y^{k+1}\| \to 0$. Since $\phi$ is coercive and $\{x^k\}$ is bounded, $X$ is thus compact. Since $\nabla h$ is locally Lipschitz, then for any $x \in X$, there exists $L_x > 0$ and $\epsilon_{1,x} > 0$, such that
$$
\|\nabla h(y)-\nabla h(z)\| \leq L_x\|y-z\|, \quad \forall y, z \in \mathbb{B}\left(x, \epsilon_{1, x}\right).
$$
According to the \cref{thm4.5} and the continuity of $\phi$, for any $x \in X$, there exist $\epsilon_{2,x} >0$, $M_x >0$ and $\theta_x \in [0,1[$, such that: 
$$
\left(1-\theta_x\right) M_x\left(\phi(y)-\phi\left(x^*\right)\right)^{-\theta_x} \operatorname{dist}\left(0, \partial_c \phi(y)\right) \leq 1, \quad \forall y \in \mathbb{B}\left(x, \epsilon_{2, x}\right) .
$$
For each $x\in X$, we take $\epsilon_x = \operatorname{min}{\epsilon_{1,x}, \epsilon_{2,x}}$. By the compactness of $X$, there exist ${z^i}_{i=1}^r \subset X$, such that $X \subset U_{i=1}^r \mathbb{B}\left(z^i, \epsilon_{z^i} /2\right).$ Taking $\epsilon = \operatorname{min}_i{\epsilon_{z^i}}$, we have for any sufficiently large $k\in \mathbb{N}^+$:
$$
x^k, y^k \in \bigcup_{i=1}^r \mathbb{B}\left(z^i, \epsilon_{z^i}\right) .
$$
Taking $\theta=\max_{i}\left\{\theta_{z^i}\right\}, K=\operatorname{min}_{i}\left\{\left(1-\theta_{z^i}\right) M_{z^i}\right\}$ and $ L=\min _i\left\{L_z{ }^i\right\}$, and using the fact that $\phi(z^i)=\phi(x^*)$ for any $z^i$, then we have:
\begin{equation}
\left|\phi\left(y^k\right)-\phi\left(x^*\right)\right|^\theta \leq K \operatorname{dist}\left(0, \partial_C\left(y^k\right)\right),
\label{equ12}
\end{equation}
and 
\begin{equation}
\left\|\nabla h\left(y^k\right)-\nabla h\left(x^k\right)\right\| \leq L\left\|y^k-x^k\right\|=L\left\|d^k\right\|.
\label{equ13}
\end{equation}
Using \cref{equ7}, we have 
\begin{equation}
\nabla h\left(x^k\right)-\nabla h\left(y^k\right) \in \partial g\left(y^k\right)-\nabla h\left(y^k\right)=\partial_C\left(\phi\left(y^k\right)\right).
\label{equ14}
\end{equation}
Combined with \cref{equ12,equ13,equ14}, we have:
\begin{equation}
\left|\phi\left(y^k\right)-\phi\left(x^*\right)\right|^\theta \leq K L\left\|d^k\right\|.
\label{equ15}
\end{equation}
\end{proof}
Since $\phi(t)=t^{1-\theta}$ is concave at $[0,+\infty[$, using \cref{def:algo4}, we have:
\begin{equation}
\begin{aligned}
& \left(\phi\left(y^k\right)-\phi\left(x^*\right)\right)^{1-\theta}-\left(\phi\left(y^{k+1}\right)-\phi\left(x^*\right)\right)^{1-\theta} \\
\geq & (1-\theta)\left(\phi\left(y^k\right)-\phi\left(x^*\right)\right)^{-\theta}\left(\phi\left(y^k\right)-\phi\left(y^{k+1}\right)\right) \\
\geq & (1-\theta)\left(\phi\left(y^k\right)-\phi\left(x^*\right)\right)^{-\theta}\left(\phi\left(x^{k+1}\right)-\phi\left(y^{k+1}\right)\right) \\
\geq & (1-\theta)\left(\phi\left(y^k\right)-\phi\left(x^*\right)\right)^{-\theta} \rho\left\|d^{k+1}\right\|^2 .
\label{equ16}
\end{aligned}
\end{equation}
We combine \cref{equ15,equ16} and apply the inequality $\frac{a^2}{b} \leq 2a-b, \forall a,b >0$, to obtain
\begin{equation}
\begin{aligned}
2\left\|d^{k+1}\right\|-\left\|d^k\right\| &\leq \frac{\left\|d^{k+1}\right\|^2}{\left\|d^k\right\|}+\left\|d^k\right\|\\
&\leq \frac{K L}{(1-\theta) \rho}\left(\left(\phi\left(y^k\right)-\phi\left(x^*\right)\right)^{1-\theta}-\left(\phi\left(y^{k+1}\right)-\phi\left(x^*\right)\right)^{1-\theta}\right),
\label{inequ17}
\end{aligned}
\end{equation}
Summing the above inequality from 0 to $N$ and letting $N \to \infty$, we have:
$$
\sum_{k=1}^{\infty}\left\|d^k\right\| \leq\left\|d^0\right\|+\frac{K L}{(1-\theta) \rho}\left(\phi\left(y^k\right)-\phi\left(x^*\right)\right)^{1-\theta}.
$$
Since $\lambda_k \leq \bar{\lambda}$, then:
$$
\sum_{k=1}^{\infty}\left\|x^{k+1}-y^k\right\|<(\bar{\lambda}-1) \sum_{k=1}^{\infty}\left\|d^k\right\|<\infty.
$$
Thus, $\left\{\sum_{k=1}^N\left\|d^k\right\|\right\}$ and $\left\{\sum_{k=1}^N\left\|x^{k+1}-y^k\right\|\right\}$ are all Cauchy sequences. For any integer $N$ and $p\in\mathbb{N}^+$,
$$
\left\|x^N-x^{N+p}\right\| \leq \sum_{k=N}^{N+p-1}\left\|x^{k+1}-x^k\right\|\leq\sum_{k=N}^{N+p-1}\left\|x^{k+1}-y^k\right\|+\sum_{k=N}^{N+p-1}\left\|d^k\right\|.
$$
Therefore, $\left\{x^k\right\}$ is also Cauchy sequence, $x^k \rightarrow x^*$, as $k \rightarrow \infty$.

\begin{remark}
In practice, we often let $\bar{\lambda_k} = \bar{\lambda} > 1$, $\forall k$, this approach is widely applied in previous research \cite{aragon2018accelerating,ferreira2021nmBDCA}. 
\end{remark}

\begin{theorem}
Assume the assumption in \cref{thm4.6} holds. Let $x^*$ be the limit point of $\{x^k\}$ and denote the Łojasiewicz exponent in \cref{thm4.6} as $\theta \in [0, 1[$. \\
\begin{enumerate}
    \item If $\theta =0$, then sequence $\{x^k\}$ converges to $x^*$ after a finite number of iterations.
    \item If $\theta \in ]0, \frac{1}{2}]$, then
    $$
    \left\|x^k-x^*\right\| \leq c_1\left(\frac{\gamma}{1+\gamma}\right)^k,
    $$
     where $\gamma=1+\frac{K^{\frac{1}{\theta}} L^{\frac{1}{\theta}}}{(1-\theta) \rho}, c_1=\bar{\lambda} \sum_{k=0}^{\infty}\left\|d^k\right\|;$
     \item If $\theta \in ] \frac{1}{2},1[$, then:
     $$
     \left\|x^k-x^*\right\| \leq c_2 k^{\frac{1-\theta}{1-2 \theta}},
    $$
    where $c_2=\bar{\lambda}\left(\frac{2 \theta-1}{\gamma(1-\theta)}\right)^{\frac{1-\theta}{1-2 \theta}}$.
\end{enumerate}
\label{thm_KLrate}
\end{theorem}

\begin{proof}
if $\theta=0$, according to \cref{equ15}, we have:
$$
\left\|d^k\right\| \geq \frac{1}{K L},
$$
which gives
$$
f\left(x^k\right)-f\left(x^{k+1}\right) \geq \rho\left\|d^k\right\|^2 \geq \frac{\rho}{K^2 L^2}.
$$
Hence, the sequence $\{x^k\}$ converges to $x^*$ after a finite number of iterations.

If $0<\theta<1$, we set $D_k=\sum_{j=k}^{\infty} \|d^k\|$. For $k>0$, $p>0$, summing \cref{inequ17} from $k-1$ to $k+p$ and letting $p \to \infty$, we have
\begin{equation}
D_k \leq\left\|d^{k-1}\right\|+\frac{K L}{(1-\theta) \rho}\left(\phi\left(y^{k-1}\right)-\phi\left(x^*\right)\right)^{1-\theta}.
\label{equ18}
\end{equation}
\end{proof}
Combining \cref{equ15} and \cref{equ18} yields
\begin{equation}
D_k \leq\left\|d^{k-1}\right\|+\frac{K^{\frac{1}{\theta}} L^{\frac{1}{\theta}}}{(1-\theta) \rho}\left\|d^{k-1}\right\|^{\frac{1-\theta}{\theta}},
\label{equ19}
\end{equation}
where we assume, without loss of generality, that for any $k$, $\|d^k\| <1$.

If $\theta \in ]0,\frac{1}{2}]$, then $\|d^k\|^{\frac{1-\theta}{\theta}} \leq \|d^k\|$. We hence obtain the following from \cref{equ19}:
$$
D_k \leq \gamma\left(D_{k-1}-D_k\right)^{\frac{1-\theta}{\theta}} .
$$
Letting $\gamma = 1+ \frac{k^{\frac{1}{\theta}}, L^{\frac{1}{\theta}}}{(1-\theta)\rho}$, we obtain
$$
D_k^{\frac{1-2 \theta}{1-\theta}}-D_{k-1}^{\frac{1-2 \theta}{1-\theta}} \geq \frac{1-2 \theta}{1-\theta} D_k^{\frac{-\theta}{1-\theta}}\left(D_k-D_{k-1}\right) \geq \frac{2 \theta-1}{\gamma(1-\theta)}.
$$
Summing these inequalities up to $\infty$ and using $\lambda_j \leq \bar{\lambda}$, we thus have:
$$
\left\|x^k-x^*\right\| \leq \bar{\lambda} D_k \leq \bar{\lambda}\left(D_0^{\frac{1-2 \theta}{1-\theta}}+\frac{2 \theta-1}{\gamma(1-\theta)} k\right)^{\frac{1-\theta}{1-2 \theta}} \leq \bar{\lambda}\left(\frac{2 \theta-1}{\gamma(1-\theta)} k\right)^{\frac{1-\theta}{1-2 \theta}} .
$$

\section{Application in Image Restoration}
\label{sec:exp}
 Image denoising is a crucial aspect of low-level image processing \cite{xie2013low,salmon2014poisson}. While numerous mathematical variational methods have been formulated specifically for addressing additive Gaussian noise, research on non-Gaussian noise remains an area of considerable potential due to its prevalent occurrence in various scientific and engineering applications. In this section, we present how the proposed IBDCA is applied to efficiently solve an important DC problem in image restoration.
\subsection{Restoring degraded images corrupted by Cauchy noise}
Cauchy noise is a type of impulse noise. Mathematically, a degraded image with Cauchy noise can be represented as \( f = u + v \), where \( u \in \mathbb{R}^m = \mathbb{R}^{m_1 \times m_2} \) is the original image and \( v \in \mathbb{R}^m \) represents some Cauchy noise. The random variable \( V \) follows a (centered) Cauchy distribution if it has a probability density function as
$$
g(v)=\frac{1}{\pi} \frac{\gamma}{\gamma^2+v^2}.
$$
The Cauchy noise removal model based on maximum a posteriori estimation (MAP) \cite{Sciacchitano2015variational} can be formulated as the following:
\begin{equation}
\min _{u \in R^m} E(u):=T V(u)+\frac{\mu}{2}\left\langle\log\left(\gamma^2+(u-f)^2\right), \mathbf{1}\right\rangle,
\label{equ20}
\end{equation}
where $T V(u)=\|\nabla u\|_{2,1}=\left\|\sqrt{\left|\nabla_x u\right|^2+\left|\nabla_y u\right|^2}\right\|_1$ denotes the discrete total variation \cite{chambolle2004algorithm}. The gradient operators representing the horizontal and vertical directions are denoted as $\nabla_x$ and $\nabla_y$, respectively, and 
 $\mu>0$ is the trade off between the total variation regularization term and the fidelity term. 

In \cite{mei2018cauchy}, a non-convex ADMM (alternating direction method of multipliers) is employed to solve \cref{equ20}, while here we utilize DC methods. \cref{equ20} can be expressed as a DC problem:
$$
\min _{u \in \mathbb{R}^m} G(u)-H(u),
$$
where we have:
\begin{equation}
\begin{aligned}
& G(u)=T V(u)+\frac{c}{2}\|u\|^2, \\
& H(u)=-\frac{\mu}{2}\left\langle\log\left(\gamma^2+(u-f)^2\right), \mathbf{1}\right\rangle+\frac{c}{2}\|u\|^2,
\label{equ21}
\end{aligned}
\end{equation}
here, $c >0$ denotes the parameter that ensures the strong convexity of $G(u)$ and $H(u)$. Firstly, we present the following lemma.

\begin{proposition}
When $c \geq \mu/\gamma^2$, the function $H$ defined in \cref{equ21} is convex. Thus, when $c > \mu/\gamma^2$, $H$ is $(c - \mu/\gamma^2)$-strongly convex.\\
\end{proposition}

\begin{proof} 
It is sufficient to consider the univariate case. We let 
$$h(t) = -\frac{\mu}{2}\log\left(\gamma^2 + t^2\right) + \frac{c}{2}t^2.$$
Then through direct calculation, we have:
$$h'(t) = -\frac{\mu t}{\gamma^2 + t^2} + ct,$$
$$h''(t) = \frac{\mu(t^2 - \gamma^2)}{(\gamma^2 + t^2)^2} + c.$$
We only need to ensure $h''(t) \leq 0$ to so that $h(t)$ is concave, i.e., 
$$
c \geq \frac{\mu(\gamma^2-t^2)}{(\gamma^2 + t^2)^2}.
$$ 
Note that the R.H.S of the above inequality reaches its maximum value at $t=0$.
\end{proof}

The DC process for solving \cref{equ20} are as follows:
\begin{equation*}
\begin{aligned}
& v^k=\nabla h\left(u^k\right)=\frac{\mu \left( u^k-f\right)}{\gamma^2+\left(u^k-f\right)^2}+c u^k . \\
& u^k=\arg \min _u T V(u)+\frac{c}{2}\|u\|^2-\left\langle v^k, u\right\rangle .
\label{equ22}
\end{aligned}
\end{equation*}
The primal-dual algorithm in \cite{chambolle2011first} can be applied to efficiently solve the $u$-subproblem. Since the objective function in the $u$-subproblem is strongly convex, the convergence rate of the primal-dual algorithm can be accelerated to $O(1/N^2)$ \cite{chambolle2011first}.\\

According to \cref{exmp4.4}, the objective function in DC problem \cref{equ20} is sub-analytic. In addition, $H$ is locally Lipschitz continuous and $E = G-H$ is coercive. Hence, the global convergence of our proposed algorithm is guaranteed by \cref{thm4.6,thm_KLrate}.

% \begin{lemma}[\cite{rudin1964principles}]
% We let $F \text{:} D \rightarrow \in \mathcal{R}^n$ be a differentiable vector function defined in the convex region $D$ in $\mathcal{R}^m$. Then, for any $x\text{,} y \in D$, there exists $\xi \in {x+(1-\theta)y: \theta \in [0,1]}$, such that:
% \begin{equation*}
% \|F(x)-F(y)\| \leq \|D F(\xi)\| \|x-y\|,
% \end{equation*}
% where $DF $ is the Jacobi matrix of $F$.
% \end{lemma}

% \begin{lemma}[\cite{johnson1985matrix}]
% Let matrix $A = (a_{ij}) \in \mathcal{L}(\mathcal{R}^m, \mathcal{R}^n)$, then:
% $$
% \|A\| \leq\left(\sum_{i=1}^n \sum_{j=1}^m\left|a_{i j}\right|^2\right)^{\frac{1}{2}}.
%\end{lemma}
% $$

\begin{figure*}[!ht]
\centering
\begin{minipage}[t]{0.3\linewidth}
\centering
\includegraphics[width=0.99\linewidth]{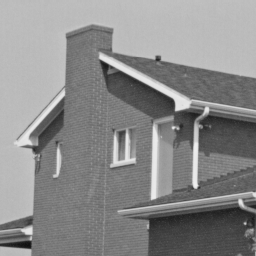}
\centerline{{\footnotesize (b) House}}
\end{minipage}%
\vspace{0.1cm}
\begin{minipage}[t]{0.3\linewidth}
\centering
\includegraphics[width=0.99\linewidth]{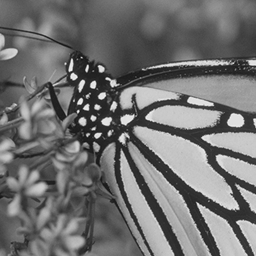}
\centerline{{\footnotesize (c) Butterfly}}
\end{minipage}\\
\begin{minipage}[t]{0.3\linewidth}
\centering
\includegraphics[width=0.99\linewidth]{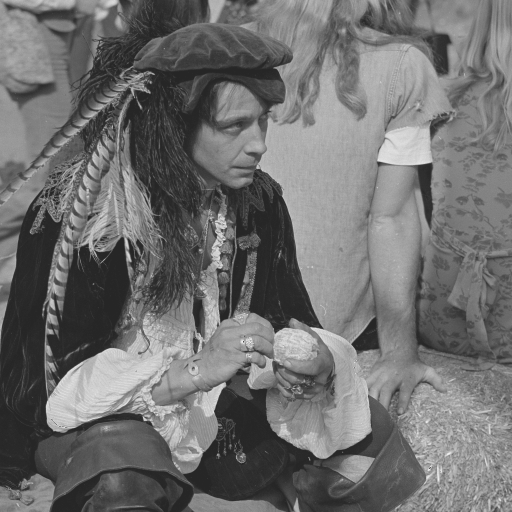}
\centerline{{\footnotesize (d) Man}}
\end{minipage}
\begin{minipage}[t]{0.3\linewidth}
\centering
\includegraphics[width=0.99\linewidth]{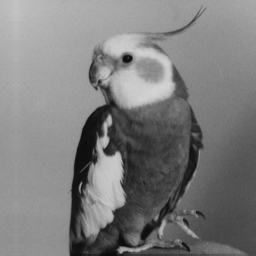}
\centerline{{\footnotesize (e) Bird}}
\end{minipage}
\begin{minipage}[t]{0.3\linewidth}
\centering
\includegraphics[width=0.99\linewidth]{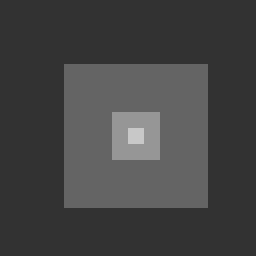}
\centerline{\footnotesize (f) Squares }
\end{minipage}
\centering
\caption{Test images: cameraman, house, butterfly, man, birds, and squares. Among these images, cameraman, house, butterfly, birds, and squares are of the size $256 \times 256$. Man is of the size $512 \times 512$.}
\centering
\label{gd}
\end{figure*}

\renewcommand{\arraystretch}{1.1}
\begin{table*}[!ht]
\small
\centering
\caption{PSNR(dB)/ReErr values of denoising on our five test images. The Cauchy noise levels are $\gamma = 3 \text{,} 5$ respectively. The CPU time (s) and number of iterations for the recovered image are also presented. We provide average results regarding each metric. Our proposed algorithm generally performs best among these comparison methods at these two different levels of Cauchy noise.}
\scalebox{0.85}{
\begin{tabular}{|lc||c|c|c|c|c|c|}
\hline 
Image & Cauchy noise level & Method & PSNR & ReErr & Time (s) & No. of Iter \\
\hline 
House & $\gamma = 3$ & ADMM & 29.21 & 0.0607 & 26.35 & 200 \\
House & $\gamma = 3$ & DCA & 28.65 & 0.0648 & 9.59 & 200 \\
House & $\gamma = 3$ & nmBDCA & 29.19 & 0.0609 & 4.64 & 71 \\
House & $\gamma = 3$ & Ours & \textbf{30.26} & \textbf{0.0538} & \textbf{2.93} & \textbf{45} \\
\hline 
House & $\gamma = 5$ & ADMM & 27.05 & 0.0779 & 23.97 & 200 \\
House & $\gamma = 5$ & DCA & 26.45 & 0.0834 & 15.15 & 122 \\
House & $\gamma = 5$ & nmBDCA & 26.85 & 0.0797 & 6.99 & 61 \\
House & $\gamma = 5$ & Ours & \textbf{27.23} & \textbf{0.0763} & \textbf{3.63} & \textbf{29} \\
\hline 
Butterfly & $\gamma = 3$ & ADMM & 26.91 & 0.0923 & 33.87 & 200 \\
Butterfly & $\gamma = 3$ & DCA & 26.61 & 0.0955 & 10.34 & 146 \\
Butterfly & $\gamma = 3$ & nmBDCA & 26.93 & 0.0920 & 4.39 & 68 \\
Butterfly & $\gamma = 3$ & Ours & \textbf{27.10} & \textbf{0.0913} & \textbf{3.37} & \textbf{51} \\
\hline 
Butterfly & $\gamma = 5$ & ADMM & 23.56 & 0.1356 & 28.50 & 200 \\
Butterfly & $\gamma = 5$ & DCA & 23.62 & 0.1347 & 17.71 & 121 \\
Butterfly & $\gamma = 5$ & nmBDCA & 23.79 & 0.1323 & 7.70 & 63 \\
Butterfly & $\gamma = 5$ & Ours & \textbf{23.83} & \textbf{0.1319} & \textbf{3.68} & \textbf{30} \\
\hline 
Man & $\gamma = 3$ & ADMM & 28.09 & 0.0827 & 117.43 & 200 \\
Man & $\gamma = 3$ & DCA & 27.79 & 0.0857 & 37.40 & 154 \\
Man & $\gamma = 3$ & nmBDCA & 28.14 & 0.0823 & 17.59 & 71 \\
Man & $\gamma = 3$ & Ours & \textbf{28.71} & \textbf{0.0771} & \textbf{11.72} & \textbf{46} \\
\hline 
Man & $\gamma = 5$ & ADMM & 25.48 & 0.1118 & 111.38 & 200 \\
Man & $\gamma = 5$ & DCA & 25.32 & 0.1138 & 66.55 & 123 \\
Man & $\gamma = 5$ & nmBDCA & 25.53 & 0.1111 & 32.32 & 63 \\
Man & $\gamma = 5$ & Ours & \textbf{25.89} & \textbf{0.1066} & \textbf{16.41} & \textbf{30} \\
\hline 
Bird & $\gamma = 3$ & ADMM & 30.53 & 0.0568 & 22.07 & 200 \\
Bird & $\gamma = 3$ & DCA & 29.55 & 0.0635 & 8.21 & 150 \\
Bird & $\gamma = 3$ & nmBDCA & 30.15 & 0.0593 & 3.88 & 69 \\
Bird & $\gamma = 3$ & Ours & \textbf{31.75} & \textbf{0.0493} & \textbf{2.64} & \textbf{46} \\
\hline 
Bird & $\gamma = 5$ & ADMM & 28.34 & 0.0731 & 24.52 & 200 \\
Bird & $\gamma = 5$ & DCA & 27.21 & 0.0833 & 14.47 & 119 \\
Bird & $\gamma = 5$ & nmBDCA & 27.79 & 0.0778 & 7.03 & 60 \\
Bird & $\gamma = 5$ & Ours & \textbf{28.67} & \textbf{0.0703} & \textbf{3.94} & \textbf{32} \\
\hline 
Squares & $\gamma = 3$ & ADMM & 33.04 & 0.0774 & 30.43 & 200 \\
Squares & $\gamma = 3$ & DCA & 31.10 & 0.0968 & 30.43 & 157 \\
Squares & $\gamma = 3$ & nmBDCA & \textbf{32.22} & \textbf{0.0852} & 5.02 & 73 \\
Squares & $\gamma = 3$ & Ours & 32.18 & 0.0856 & \textbf{4.47} & 53 \\
\hline 
Squares & $\gamma = 5$ & ADMM & 31.36 & 0.0940 & 31.63 & 200 \\
Squares & $\gamma = 5$ & DCA & 29.09 & 0.1221 & 19.68 & 126 \\
Squares & $\gamma = 5$ & nmBDCA & 30.10 & 0.1086 & 9.52 & 64 \\
Squares & $\gamma = 5$ & Ours & \textbf{31.61} & \textbf{0.0914} & \textbf{5.10} & \textbf{32} \\
\hline 
Average & $\gamma = 3$ & ADMM & 29.56 & 0.0740 & 46.03 & 200 \\
Average & $\gamma = 3$ & DCA & 28.74 & 0.0813 & 19.19 & 162 \\
Average & $\gamma = 3$ & nmBDCA & 29.33 & 0.0759 & 7.10 & 71 \\
Average & $\gamma = 3$ & Ours & \textbf{30.00} & \textbf{0.0714} & \textbf{5.03} & \textbf{48} \\
\hline 
Average & $\gamma = 5$ & ADMM & 25.48 & 0.1118 & 44.00 & 200 \\
Average & $\gamma = 5$ & DCA & 25.32 & 0.1138 & 26.71 & 122 \\
Average & $\gamma = 5$ & nmBDCA & 25.71 & 0.1019 & 32.32 & 63 \\
Average & $\gamma = 5$ & Ours & \textbf{27.44} & \textbf{0.0953} & \textbf{6.55} & \textbf{31} \\
\hline
\end{tabular}
}
\label{selectedimg_cauchy}
\end{table*}

%%%  gamma 3 / house %%%
\begin{figure*}[ht]
\centering
\begin{minipage}[t]{0.28\linewidth}
\centering
\includegraphics[width=0.99\linewidth]{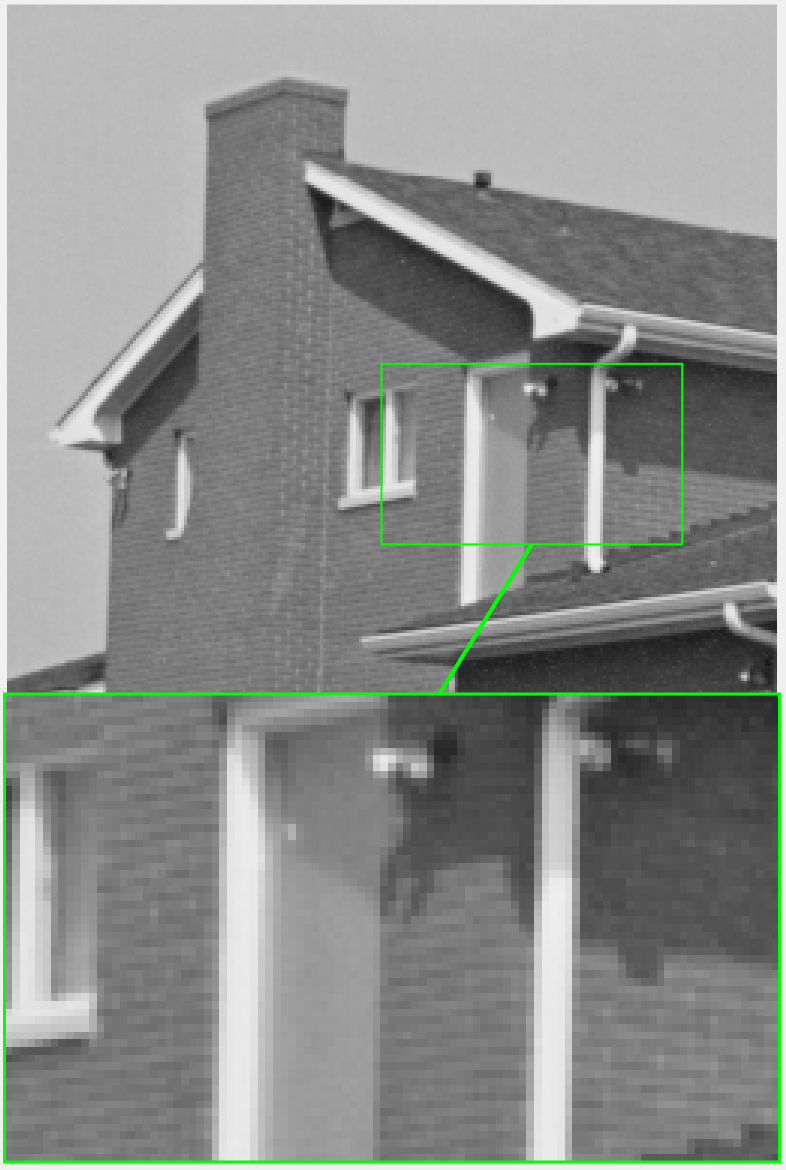}
\centerline{\footnotesize (a) House}
\centerline{\footnotesize Groundtruth}
\end{minipage}%
\vspace{0.1cm}
\begin{minipage}[t]{0.28\linewidth}
\centering
\includegraphics[width=0.99\linewidth]{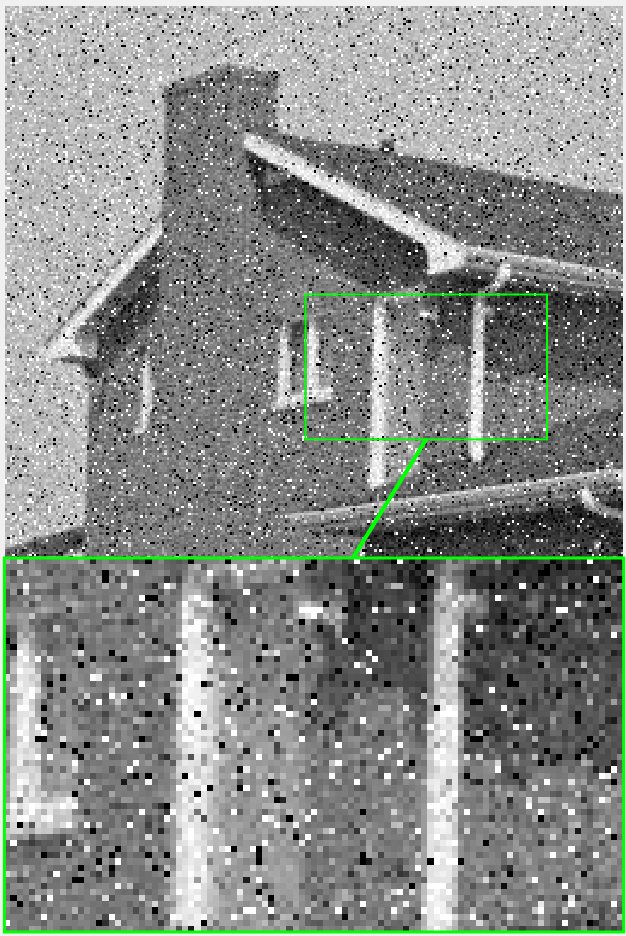}
\centerline{{\footnotesize (b) Noisy image}}
\centerline{{\footnotesize PSNR:16.79}}
\end{minipage}%
\vspace{0.1cm}
\begin{minipage}[t]{0.28\linewidth}
\centering
\includegraphics[width=0.99\linewidth]{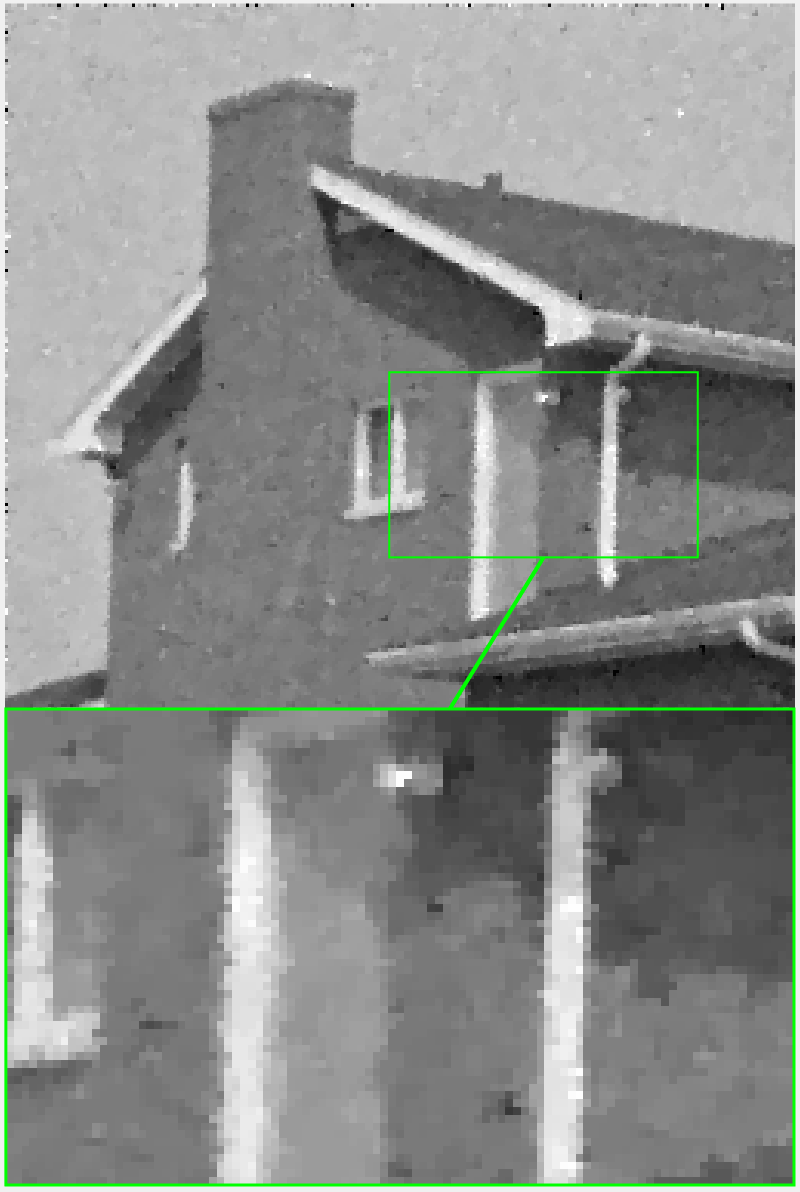}
\centerline{{\footnotesize (c) ADMM\cite{mei2018cauchy}}}
\centerline{{\footnotesize PSNR:29.21}}
\end{minipage}\\
%\vspace{+0.8cm}
\begin{minipage}[t]{0.28\linewidth}
\centering
\includegraphics[width=0.99\linewidth]{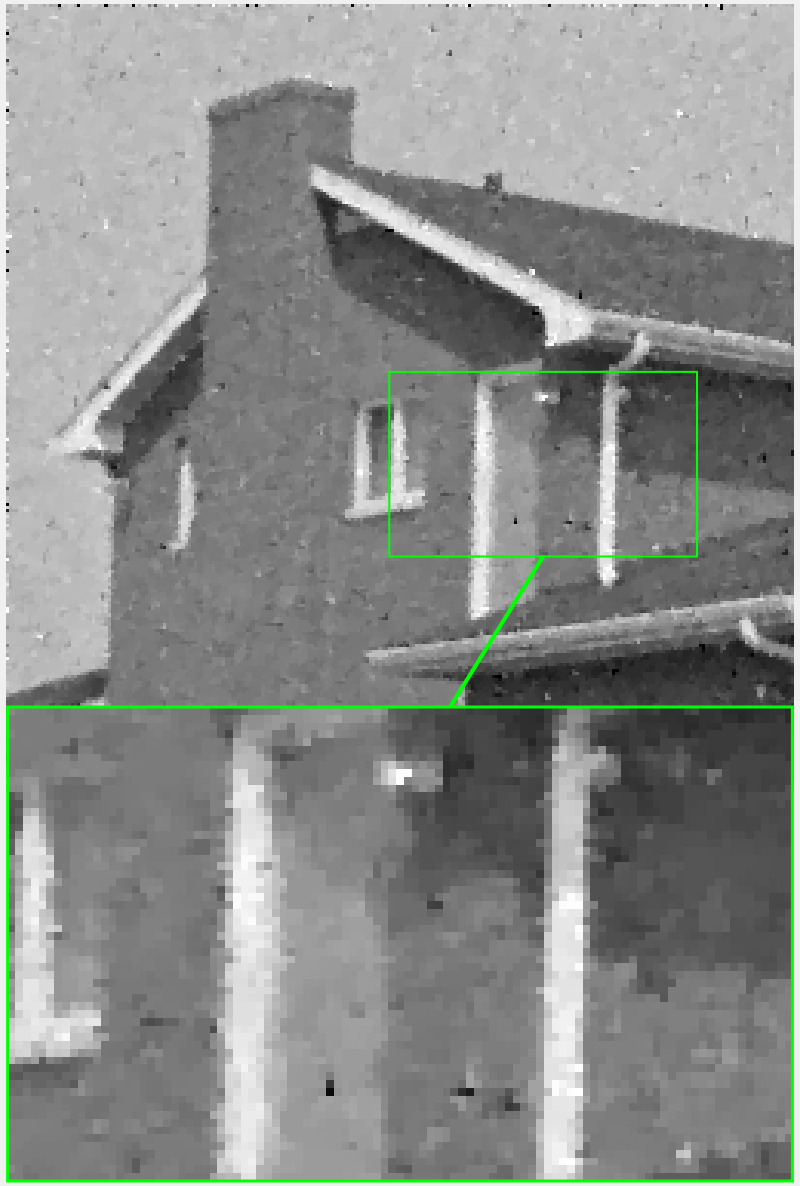}
\centerline{\footnotesize (d) DCA\cite{tao1996numerical}}
\centerline{\footnotesize PSNR:28.65}
\end{minipage}
\begin{minipage}[t]{0.28\linewidth}
\centering
\includegraphics[width=0.99\linewidth]{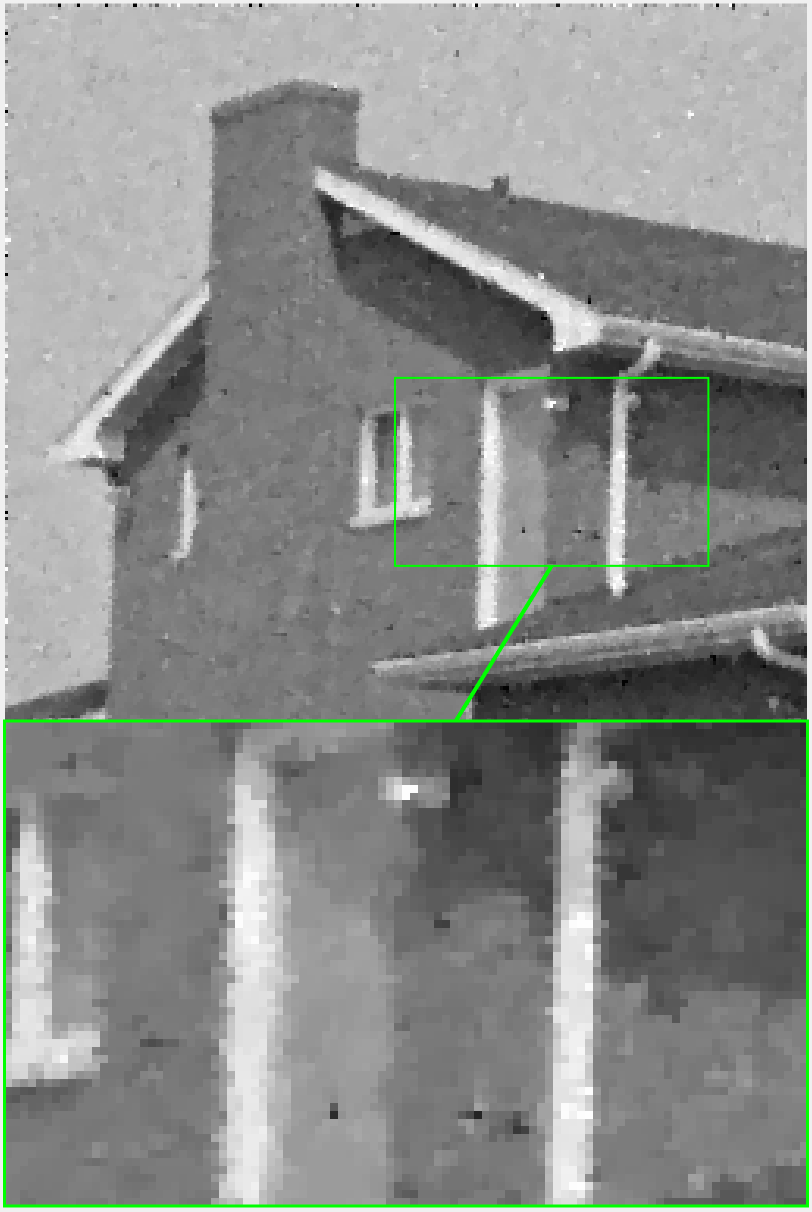}
\centerline{\footnotesize (e) nmBDCA\cite{ferreira2024boosted}}
\centerline{\footnotesize PSNR:29.19}
\end{minipage}
\begin{minipage}[t]{0.28\linewidth}
\centering
\includegraphics[width=0.99\linewidth]{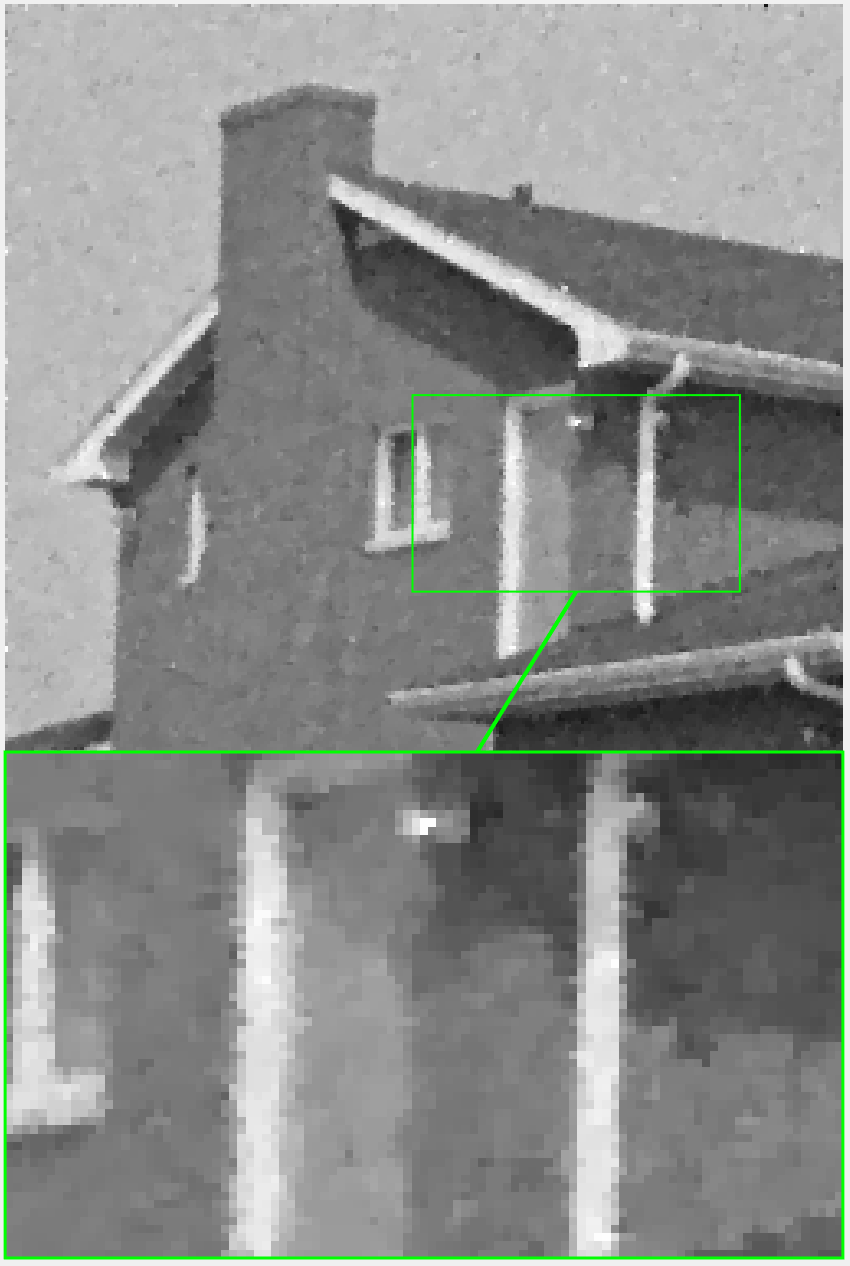}
\centerline{\footnotesize (f) Ours }
\centerline{\footnotesize PSNR:30.26}
\end{minipage}
\caption{Visual comparisons on the selected restored image, which features Cauchy noise with $\gamma = 3$. Zoom in for better visualization.}
\centering
\label{house}
\end{figure*}

%%%  gamma 3 / man %%%
\begin{figure*}[!ht]
\centering
\begin{minipage}[t]{0.28\linewidth}
\centering
\includegraphics[width=0.99\linewidth]{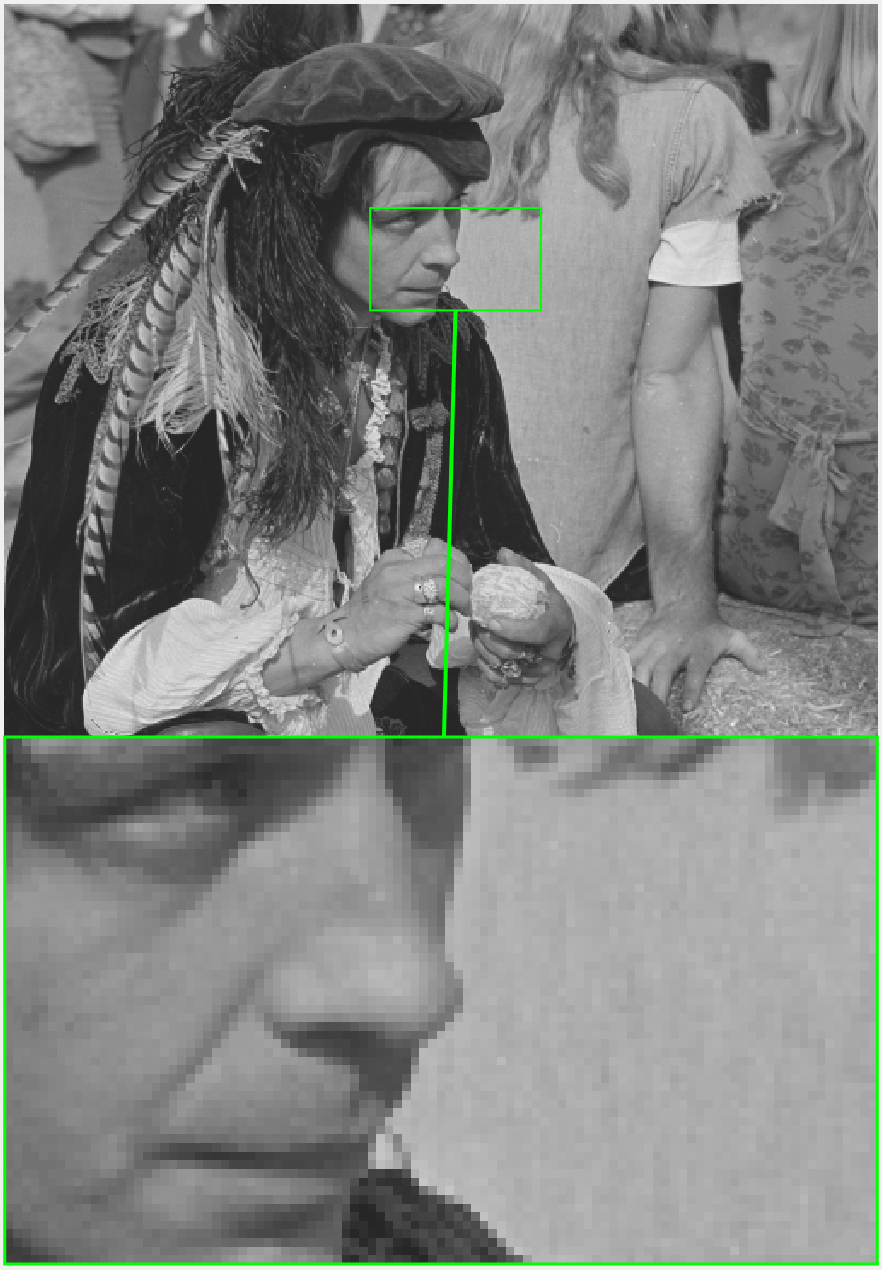}
\centerline{\footnotesize (a) Man}
\centerline{\footnotesize Groundtruth}
\end{minipage}%
\vspace{0.1cm}
\begin{minipage}[t]{0.28\linewidth}
\centering
\includegraphics[width=0.99\linewidth]{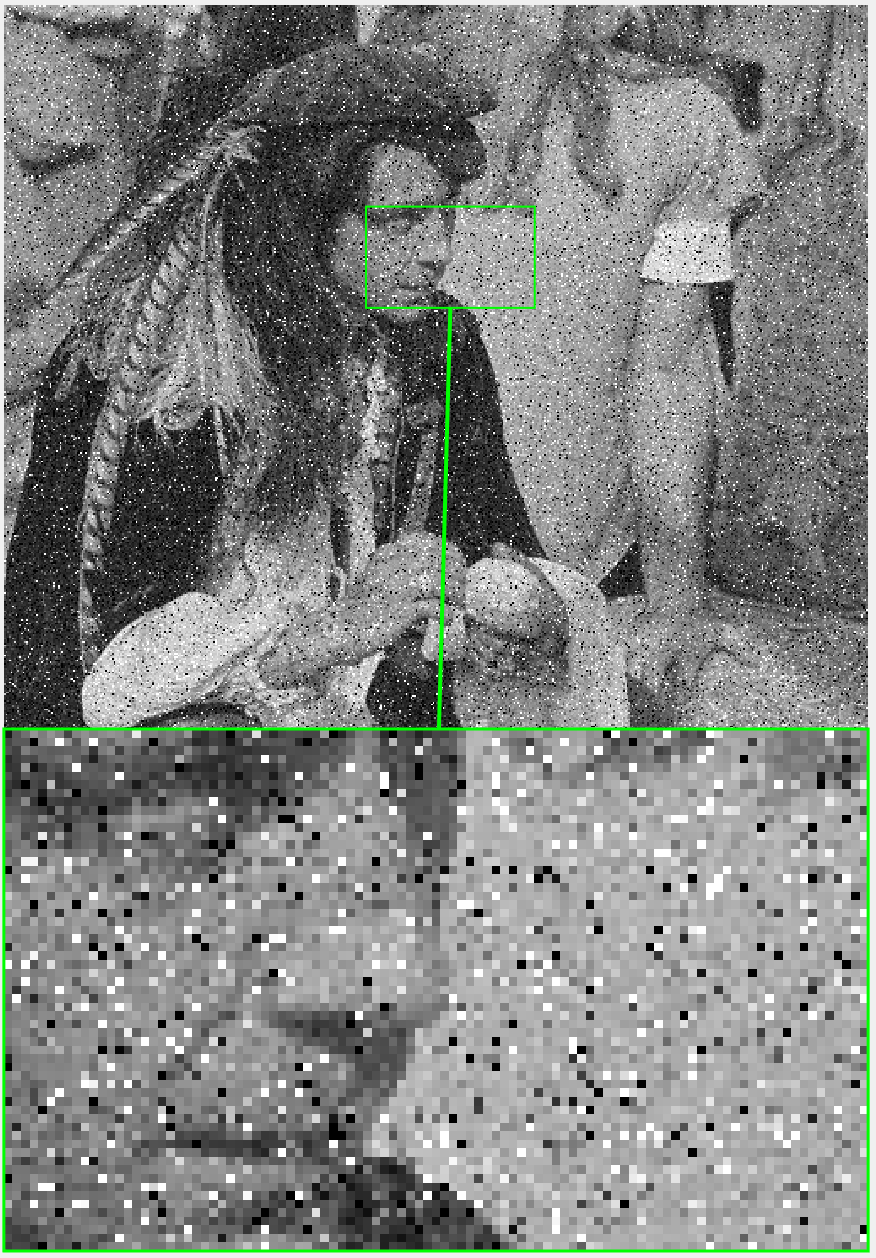}
\centerline{{\footnotesize (b) Noisy image}}
\centerline{{\footnotesize PSNR:16.74}}
\end{minipage}%
\vspace{0.1cm}
\begin{minipage}[t]{0.28\linewidth}
\centering
\includegraphics[width=0.99\linewidth]{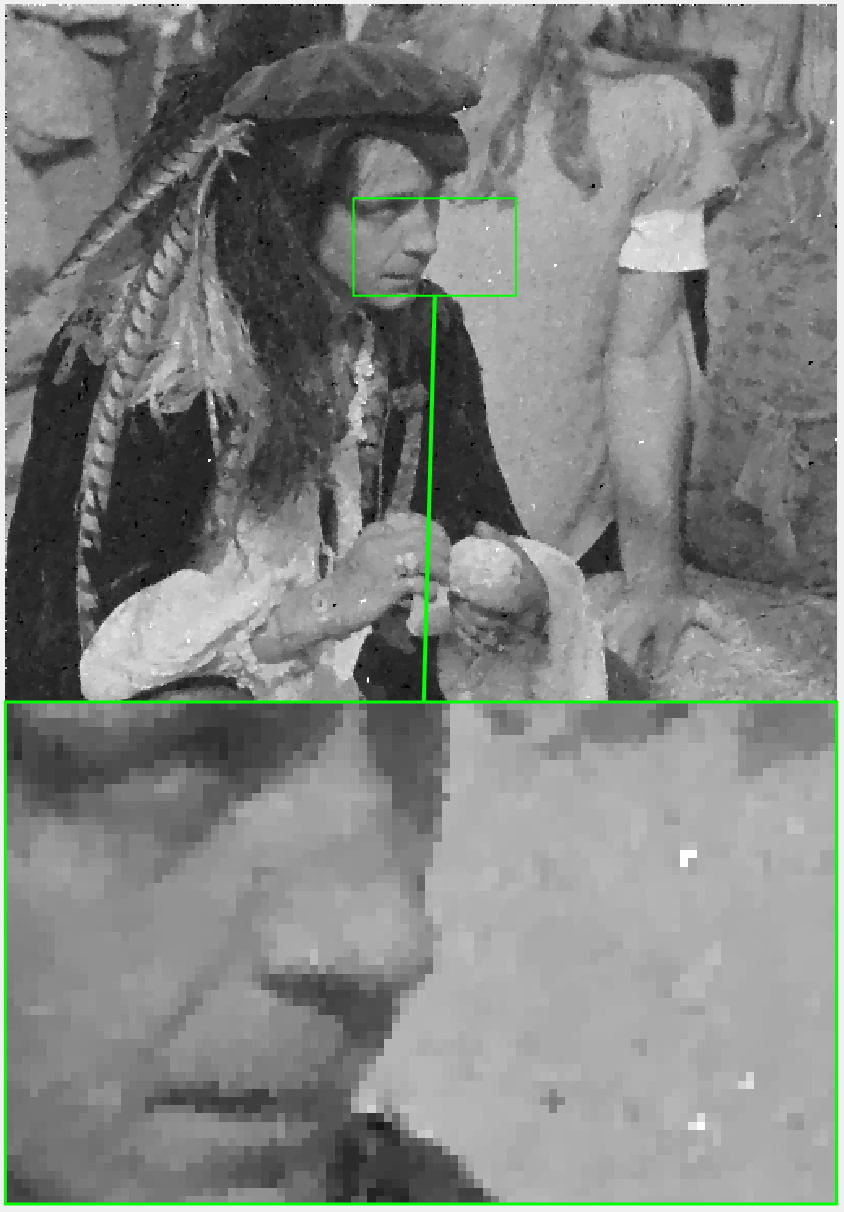}
\centerline{{\footnotesize (c) ADMM\cite{mei2018cauchy}}}
\centerline{{\footnotesize PSNR:28.09}}
\end{minipage}\\
%\vspace{+0.8cm}
\begin{minipage}[t]{0.28\linewidth}
\centering
\includegraphics[width=0.99\linewidth]{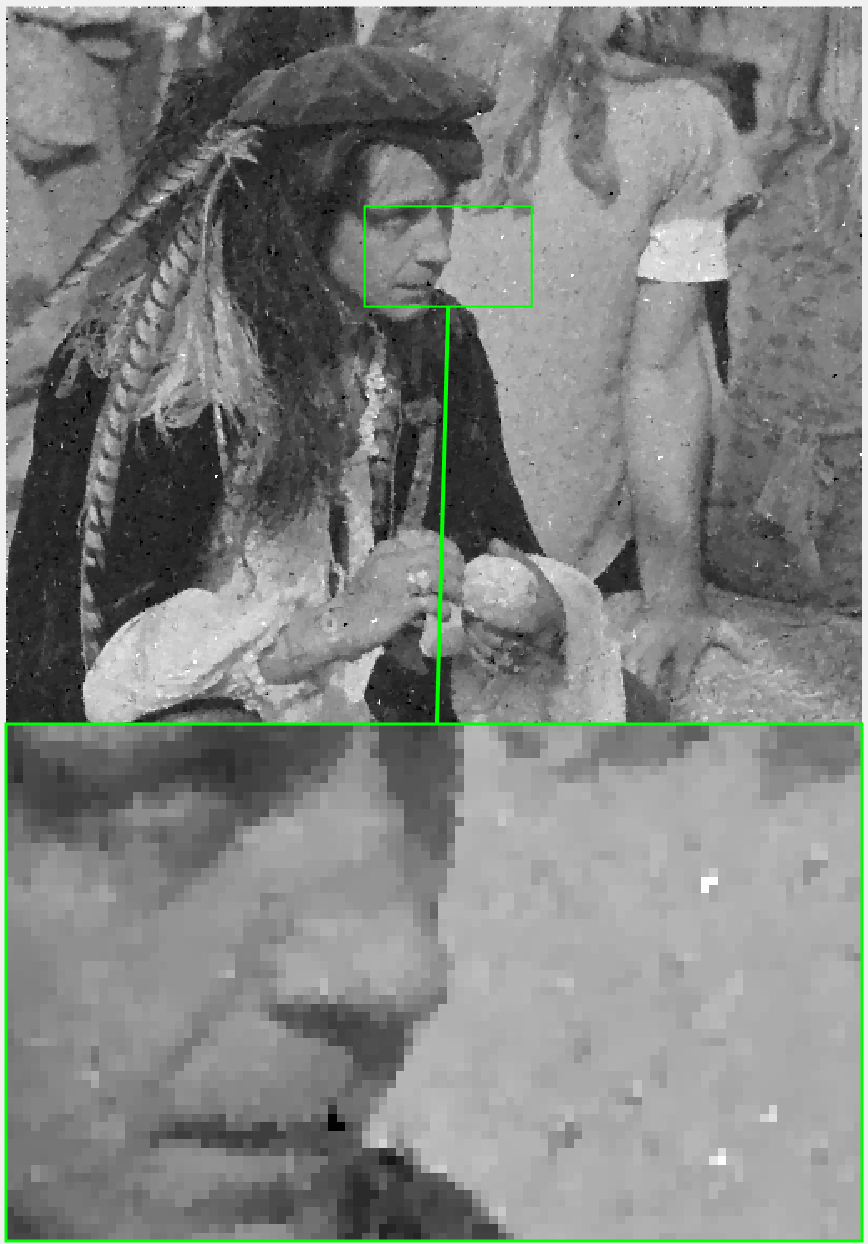}
\centerline{\footnotesize (d) DCA\cite{tao1996numerical}}
\centerline{\footnotesize PSNR:27.79}
\end{minipage}
\begin{minipage}[t]{0.28\linewidth}
\centering
\includegraphics[width=0.99\linewidth]{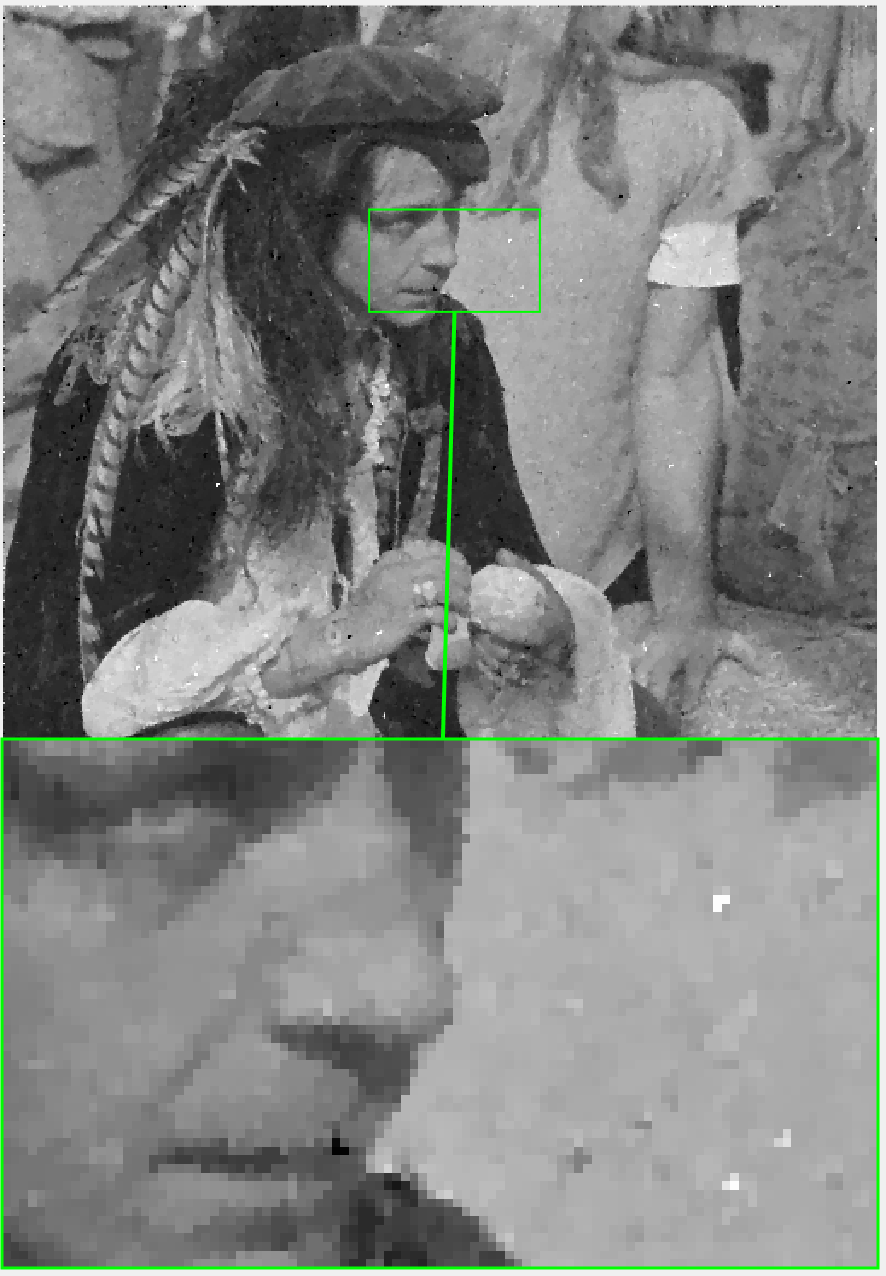}
\centerline{\footnotesize (e) nmBDCA\cite{ferreira2024boosted}}
\centerline{\footnotesize PSNR:28.14}
\end{minipage}
\begin{minipage}[t]{0.28\linewidth}
\centering
\includegraphics[width=0.99\linewidth]{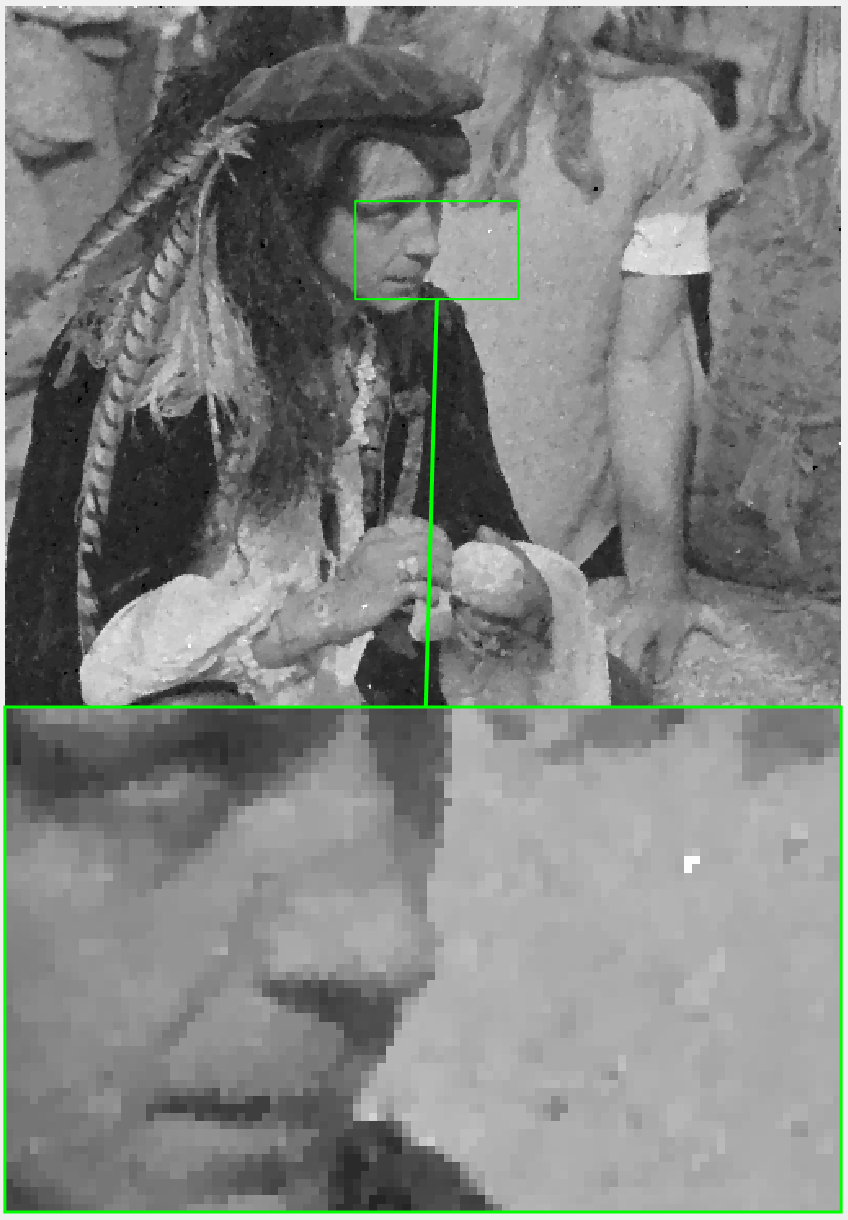}
\centerline{\footnotesize (f) Ours }
\centerline{\footnotesize PSNR:28.71}
\end{minipage}
\caption{Visual comparisons on the selected restored image, which features Cauchy noise with $\gamma = 3$. Zoom in for better visualization.}
\centering
\label{man}
\end{figure*}

%%%  gamma 3 / bird %%%
\begin{figure*}[!ht]
\centering
\begin{minipage}[t]{0.28\linewidth}
\centering
\includegraphics[width=0.99\linewidth]{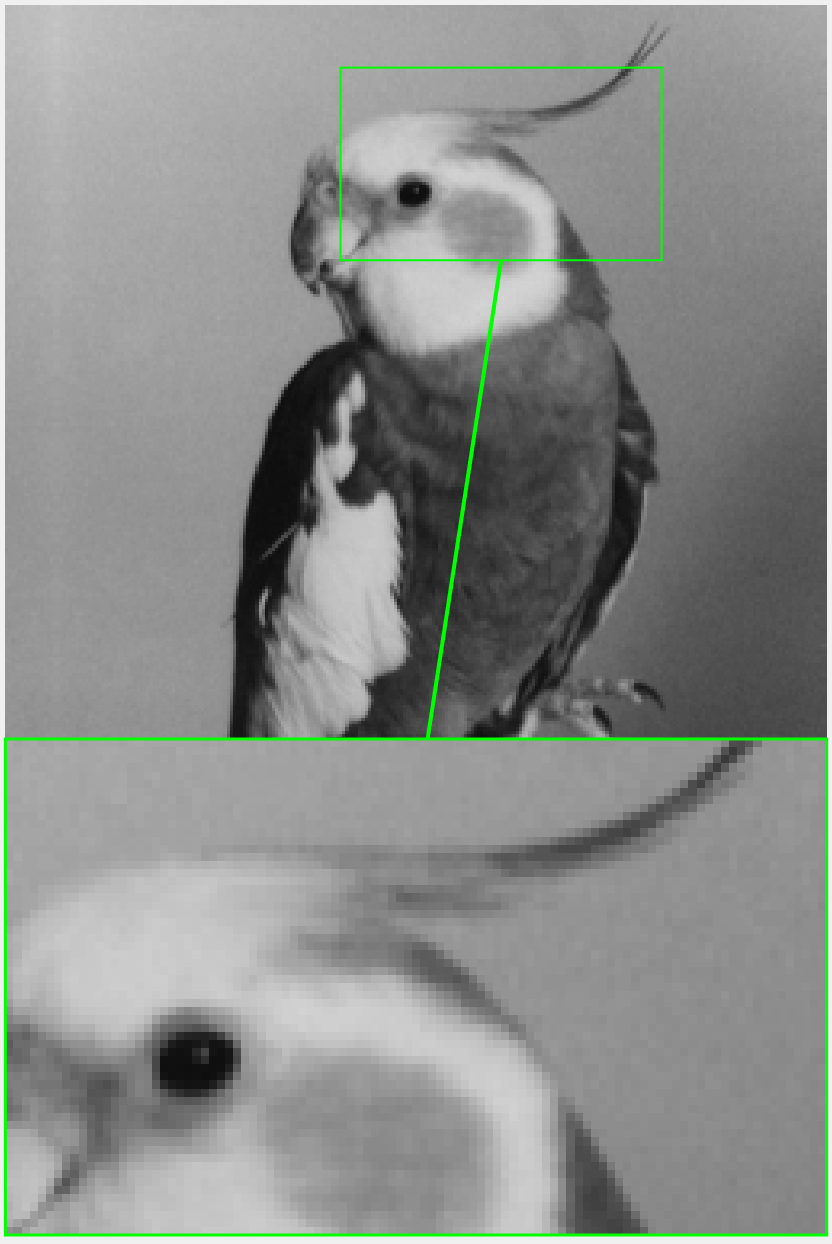}
\centerline{\footnotesize (a) Bird}
\centerline{\footnotesize Groundtruth}
\end{minipage}%
\vspace{0.1cm}
\begin{minipage}[t]{0.28\linewidth}
\centering
\includegraphics[width=0.99\linewidth]{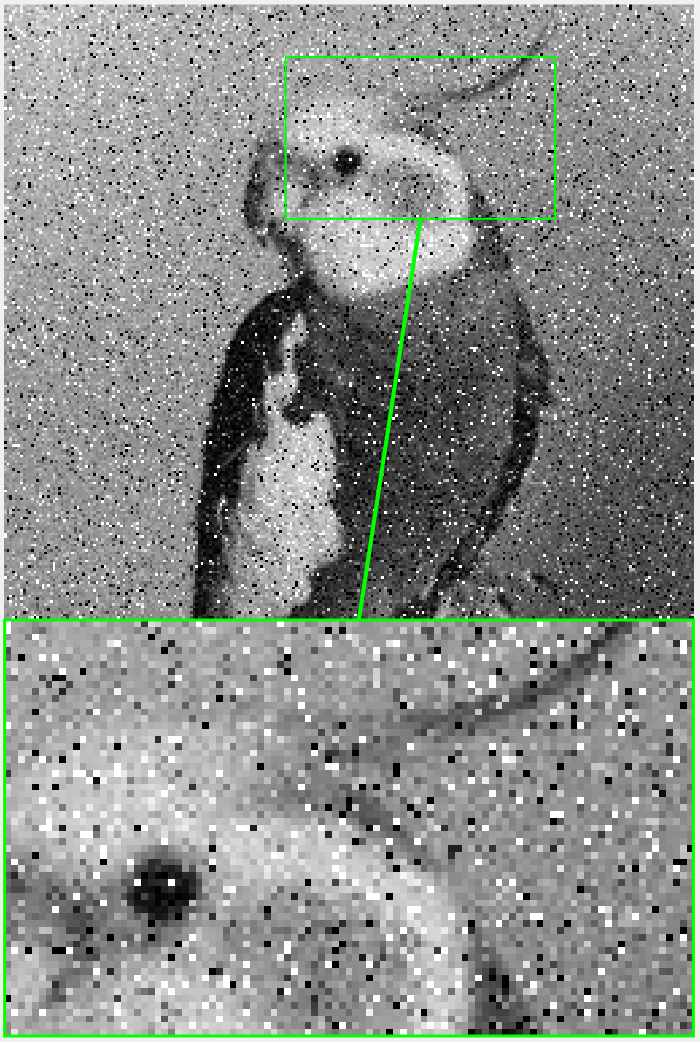}
\centerline{{\footnotesize (b) Noisy image}}
\centerline{{\footnotesize PSNR:16.77}}
\end{minipage}%
\vspace{0.1cm}
\begin{minipage}[t]{0.28\linewidth}
\centering
\includegraphics[width=0.99\linewidth]{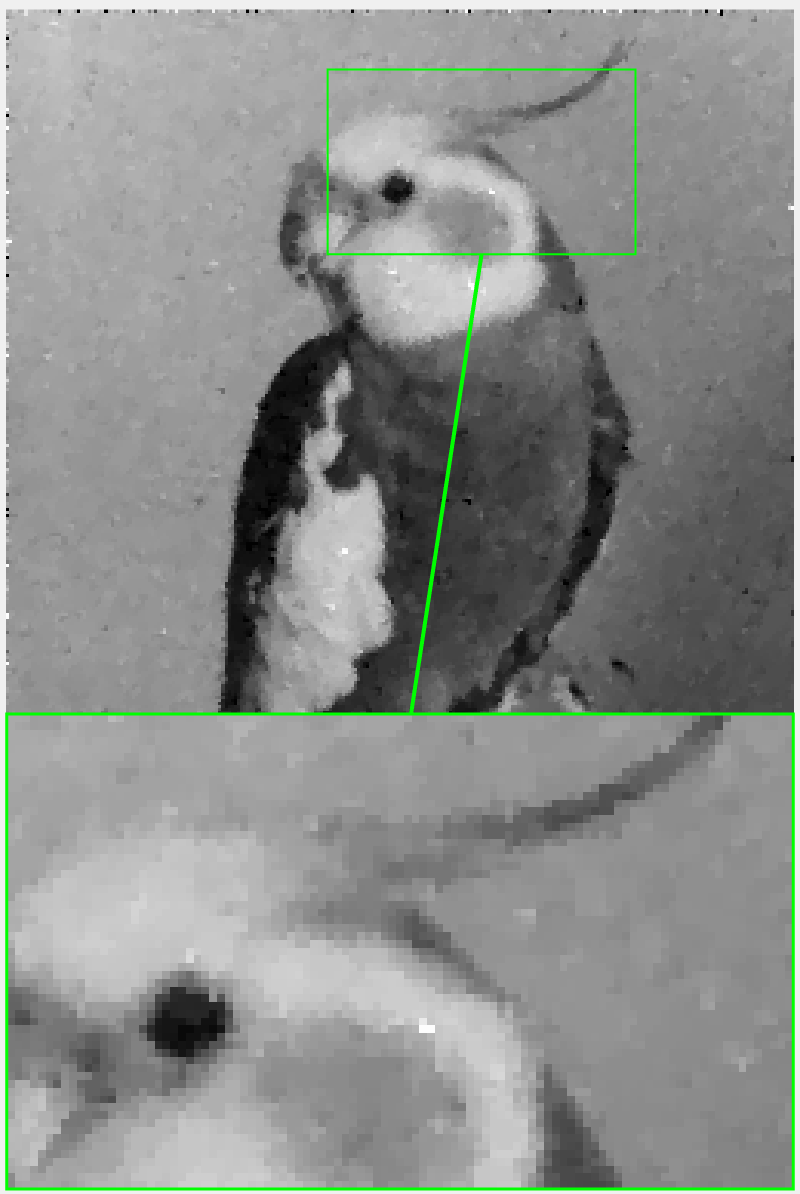}
\centerline{{\footnotesize (c) ADMM\cite{mei2018cauchy}}}
\centerline{{\footnotesize PSNR:30.53}}
\end{minipage}\\
%\vspace{+0.8cm}
\begin{minipage}[t]{0.28\linewidth}
\centering
\includegraphics[width=0.99\linewidth]{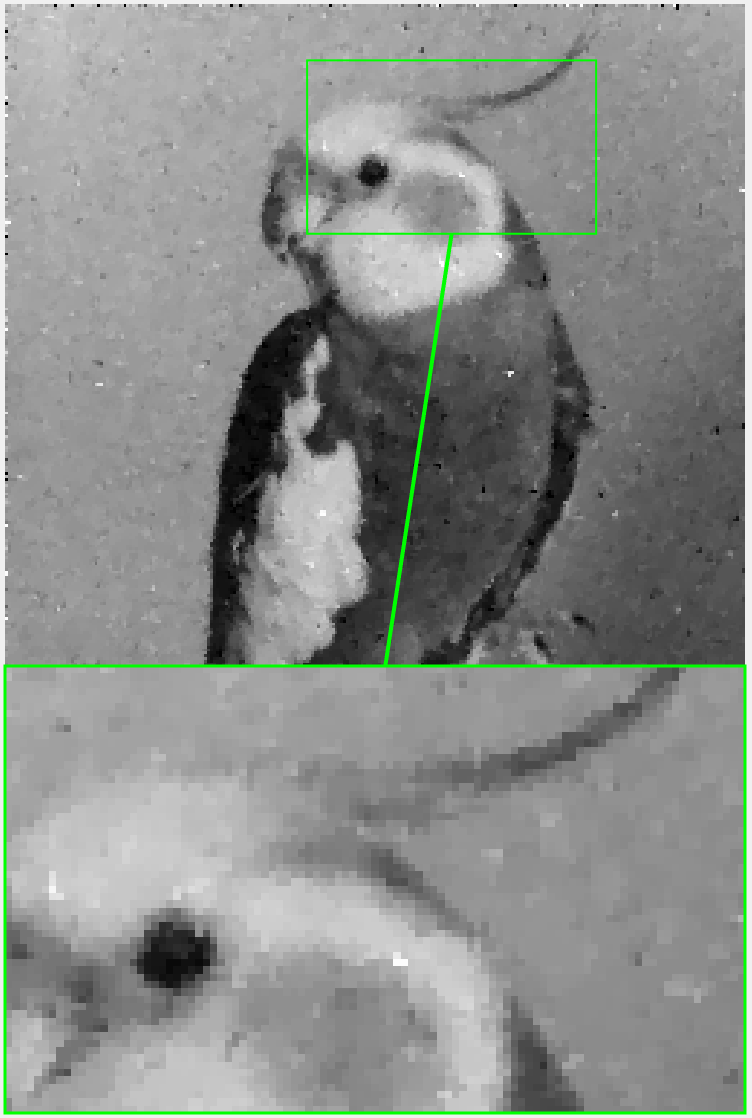}
\centerline{\footnotesize (d) DCA\cite{tao1996numerical}}
\centerline{\footnotesize PSNR:29.55}
\end{minipage}
\begin{minipage}[t]{0.28\linewidth}
\centering
\includegraphics[width=0.99\linewidth]{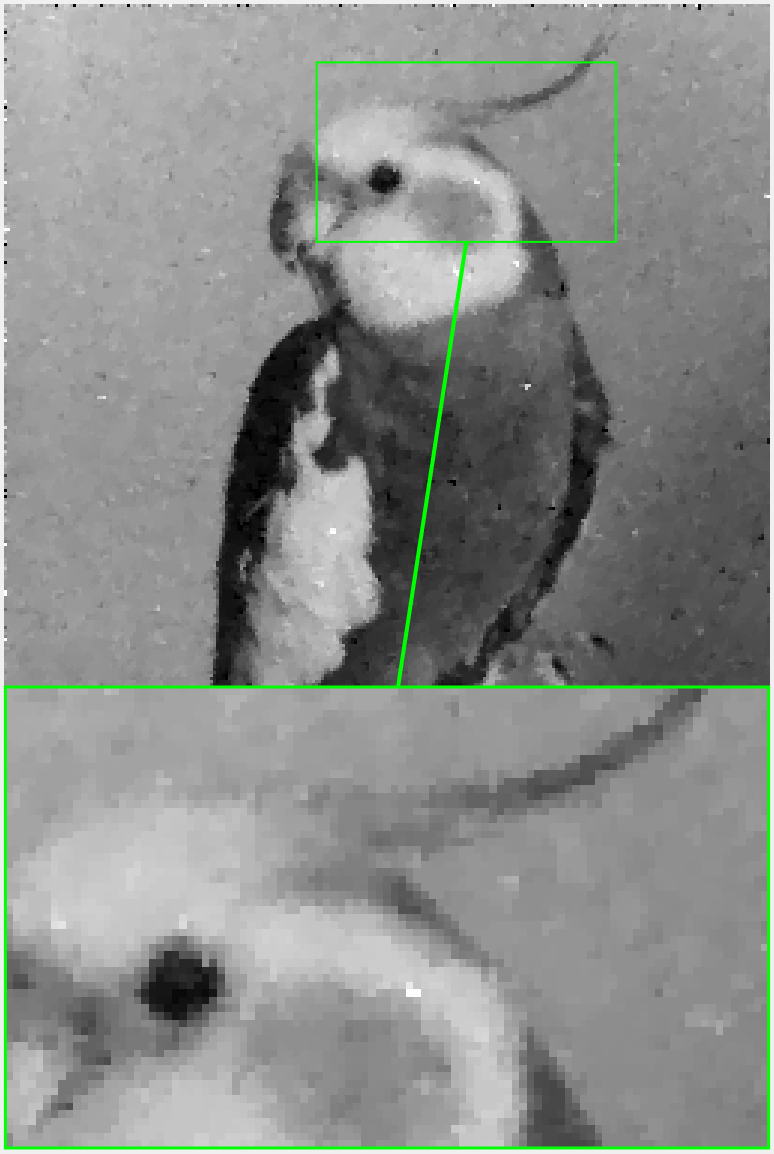}
\centerline{\footnotesize (e) nmBDCA\cite{ferreira2024boosted}}
\centerline{\footnotesize PSNR:30.15}
\end{minipage}
\begin{minipage}[t]{0.28\linewidth}
\centering
\includegraphics[width=0.99\linewidth]{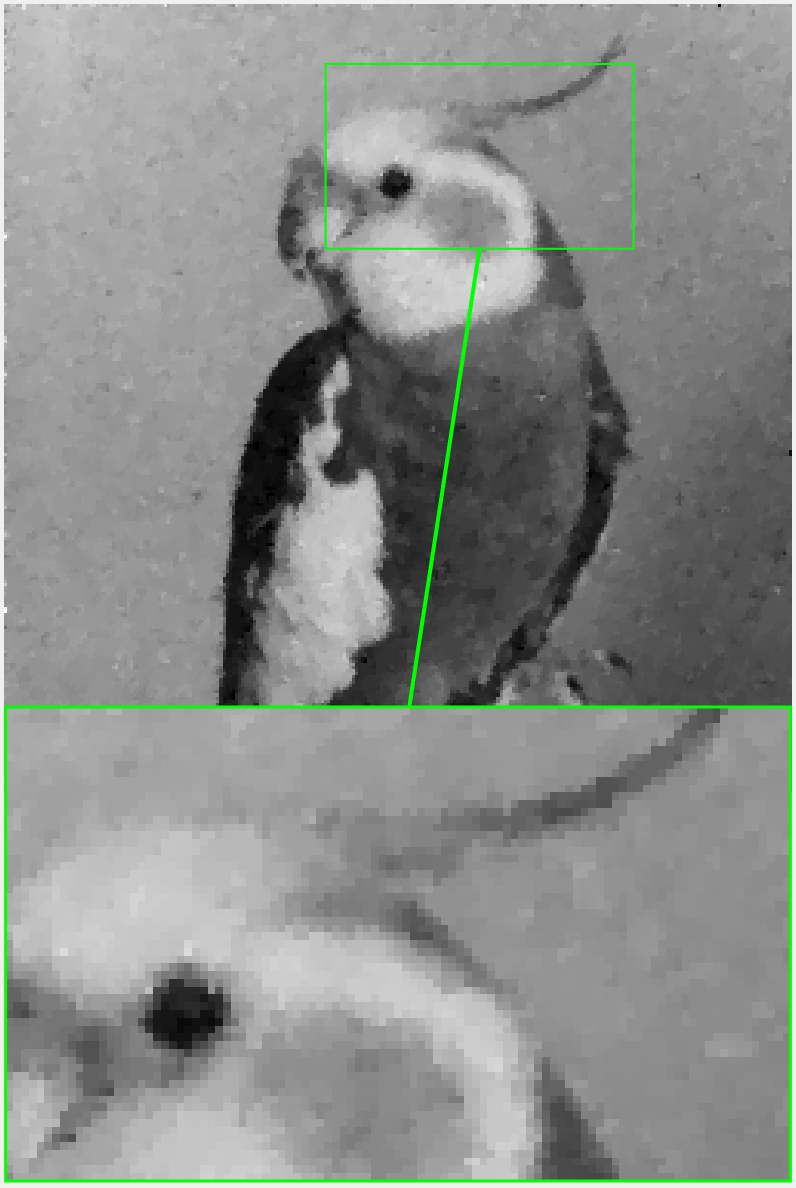}
\centerline{\footnotesize (f) Ours }
\centerline{\footnotesize PSNR:31.75}
\end{minipage}
\caption{Visual comparisons on the selected restored image, which features Cauchy noise with $\gamma = 3$. Zoom in for better visualization.}
\centering
\label{bird}
\end{figure*}

%%%  gamma 5 / house %%%
\begin{figure*}[ht]
\centering
\begin{minipage}[t]{0.28\linewidth}
\centering
\includegraphics[width=0.99\linewidth]{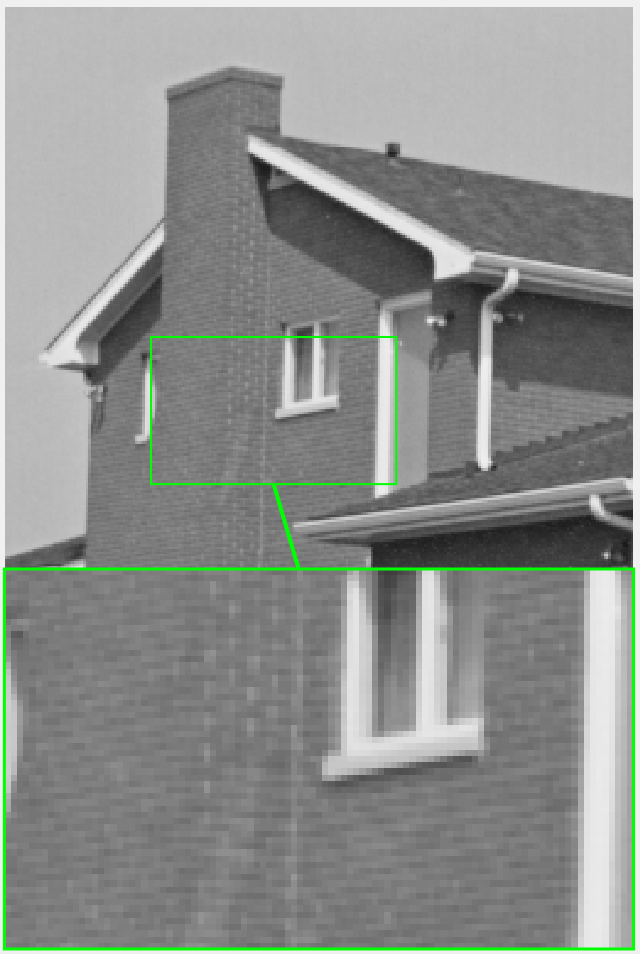}
\centerline{\footnotesize (a) House}
\centerline{\footnotesize Groundtruth}
\end{minipage}%
\vspace{0.1cm}
\begin{minipage}[t]{0.28\linewidth}
\centering
\includegraphics[width=0.99\linewidth]{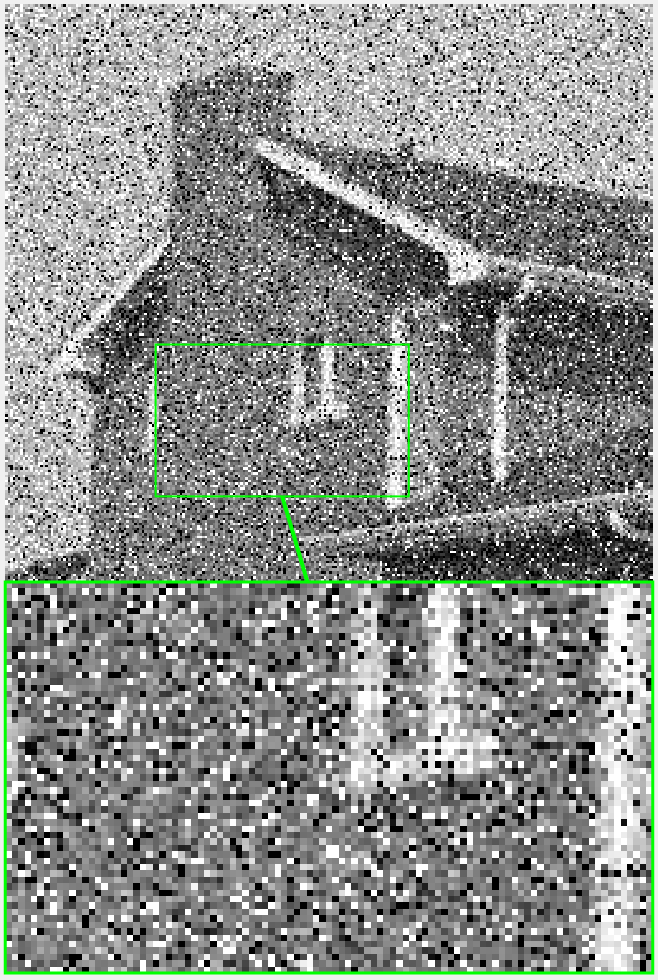}
\centerline{{\footnotesize (b) Noisy image}}
\centerline{{\footnotesize PSNR:12.73}}
\end{minipage}%
\vspace{0.1cm}
\begin{minipage}[t]{0.28\linewidth}
\centering
\includegraphics[width=0.99\linewidth]{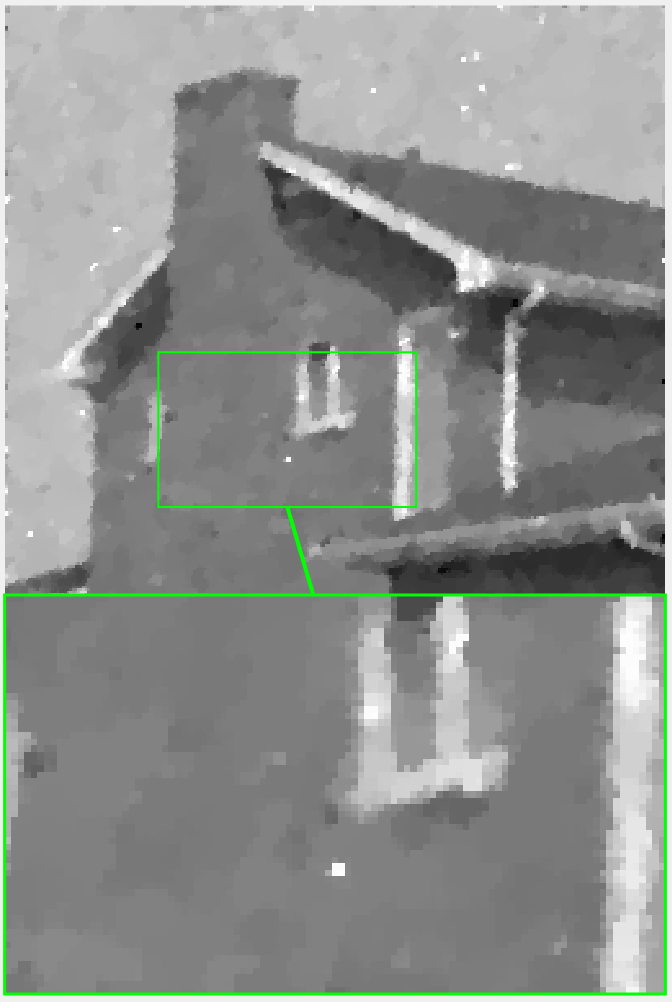}
\centerline{{\footnotesize (c) ADMM\cite{mei2018cauchy}}}
\centerline{{\footnotesize PSNR:27.05}}
\end{minipage}\\
\begin{minipage}[t]{0.28\linewidth}
\centering
\includegraphics[width=0.99\linewidth]{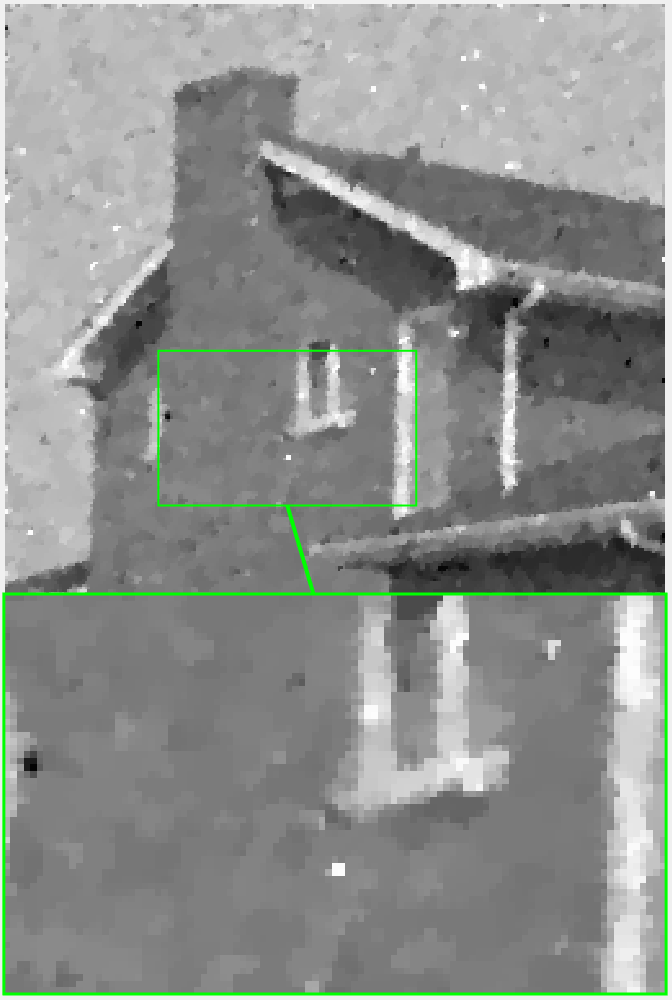}
\centerline{\footnotesize (d) DCA\cite{tao1996numerical}}
\centerline{\footnotesize PSNR:26.45}
\end{minipage}
\begin{minipage}[t]{0.28\linewidth}
\centering
\includegraphics[width=0.99\linewidth]{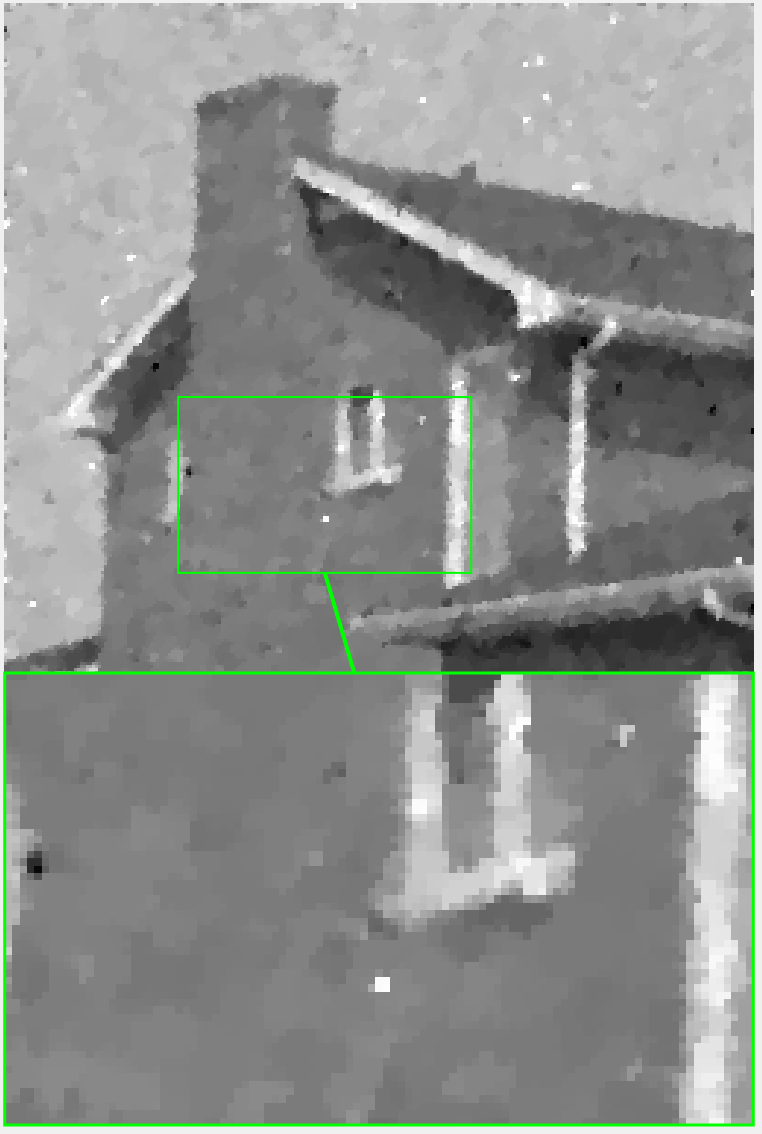}
\centerline{\footnotesize (e) nmBDCA\cite{ferreira2024boosted}}
\centerline{\footnotesize PSNR:26.85}
\end{minipage}
\begin{minipage}[t]{0.28\linewidth}
\centering
\includegraphics[width=0.99\linewidth]{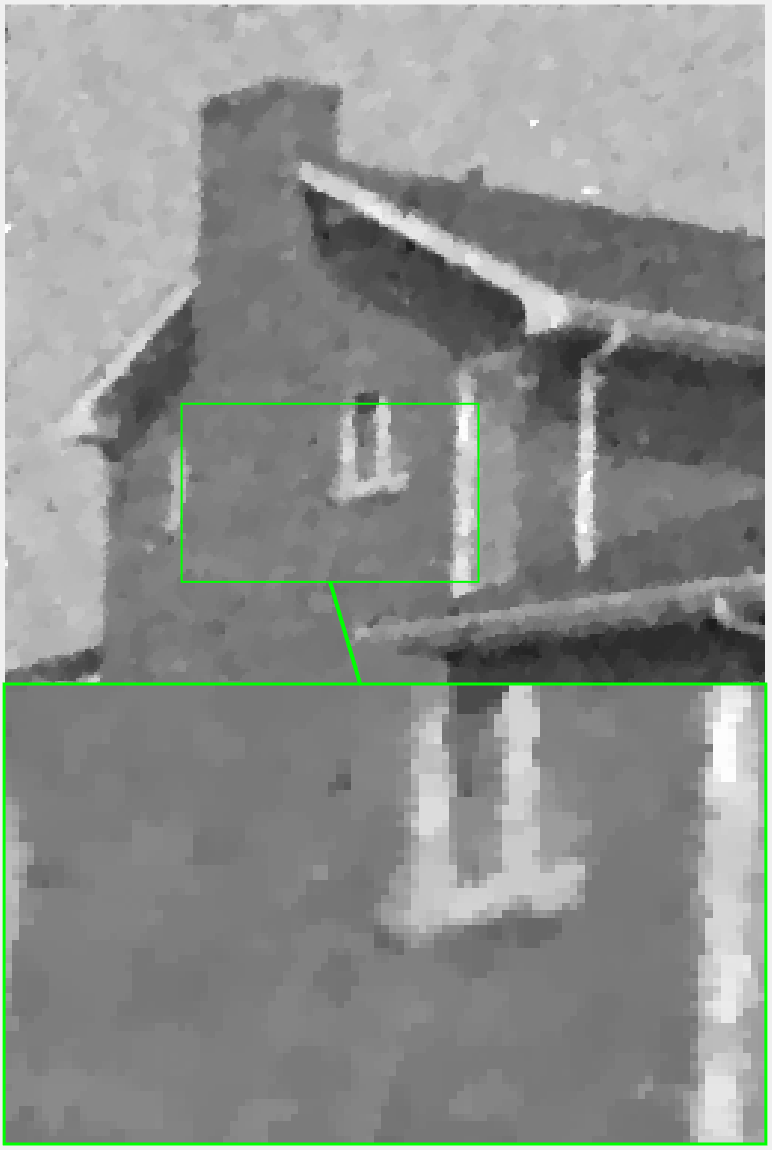}
\centerline{\footnotesize (f) Ours }
\centerline{\footnotesize PSNR:27.23}
\end{minipage}
\caption{Visual comparisons on the selected restored image, which features Cauchy noise with $\gamma = 5$. Zoom in for better visualization.}
\centering
\label{house5}
\end{figure*}

%%%  gamma 5 / man %%%
\begin{figure*}[!ht]
\centering
\begin{minipage}[t]{0.28\linewidth}
\centering
\includegraphics[width=0.99\linewidth]{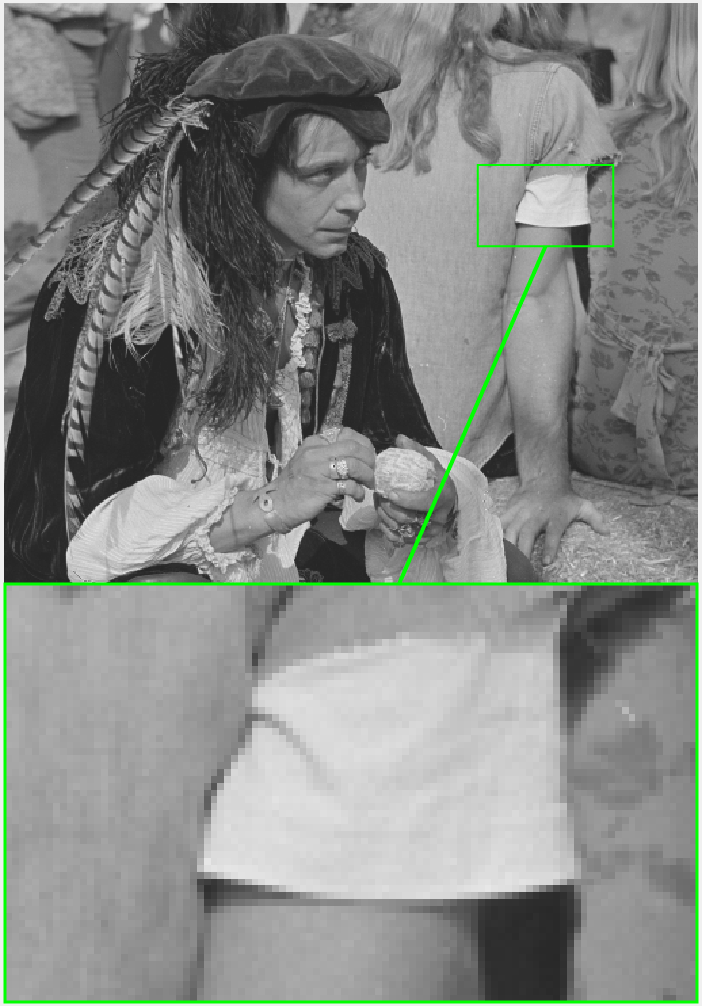}
\centerline{\footnotesize (a) Man}
\centerline{\footnotesize Groundtruth}
\end{minipage}%
\vspace{0.1cm}
\begin{minipage}[t]{0.28\linewidth}
\centering
\includegraphics[width=0.99\linewidth]{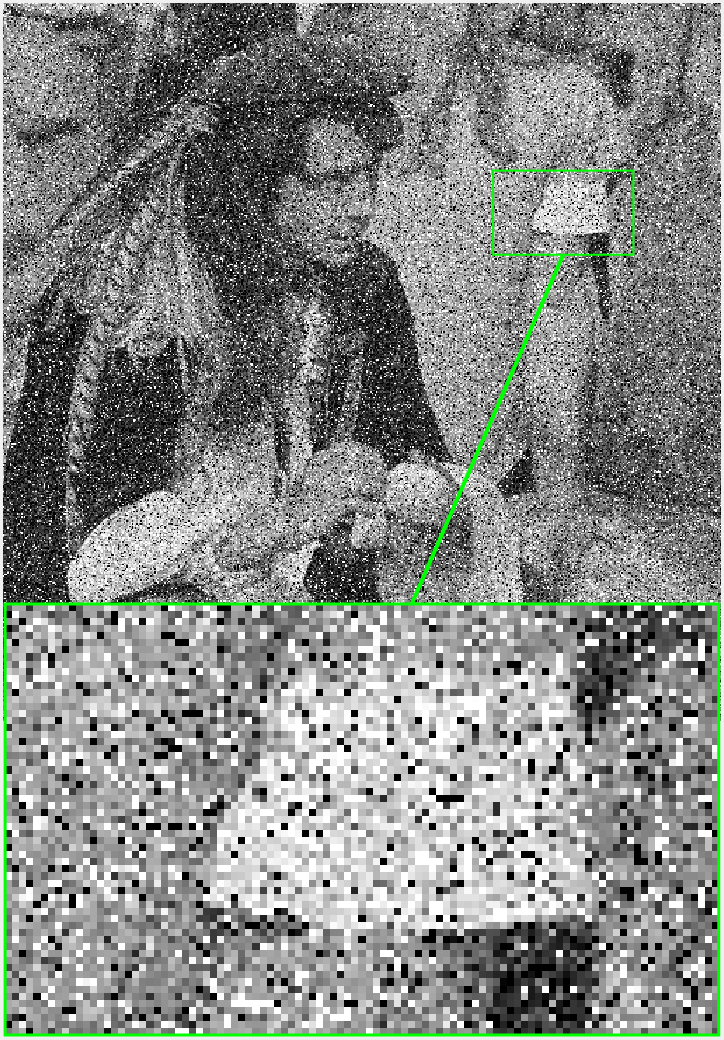}
\centerline{{\footnotesize (b) Noisy image}}
\centerline{{\footnotesize PSNR:12.67}}
\end{minipage}%
\vspace{0.1cm}
\begin{minipage}[t]{0.28\linewidth}
\centering
\includegraphics[width=0.99\linewidth]{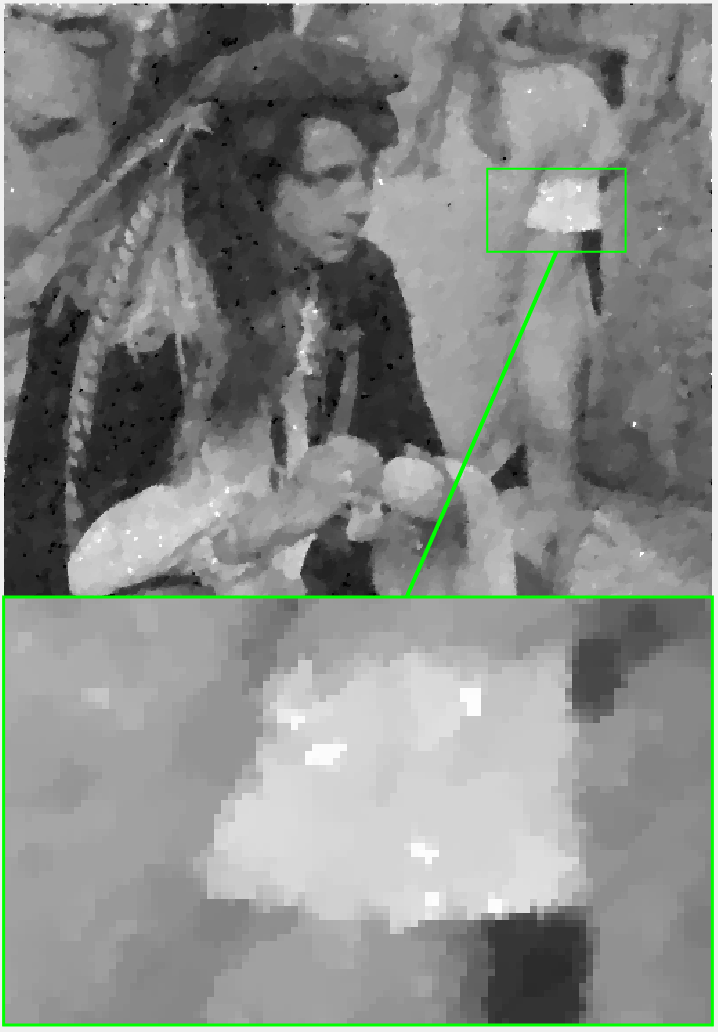}
\centerline{{\footnotesize (c) ADMM\cite{mei2018cauchy}}}
\centerline{{\footnotesize PSNR:25.48}}
\end{minipage}\\
\begin{minipage}[t]{0.28\linewidth}
\centering
\includegraphics[width=0.99\linewidth]{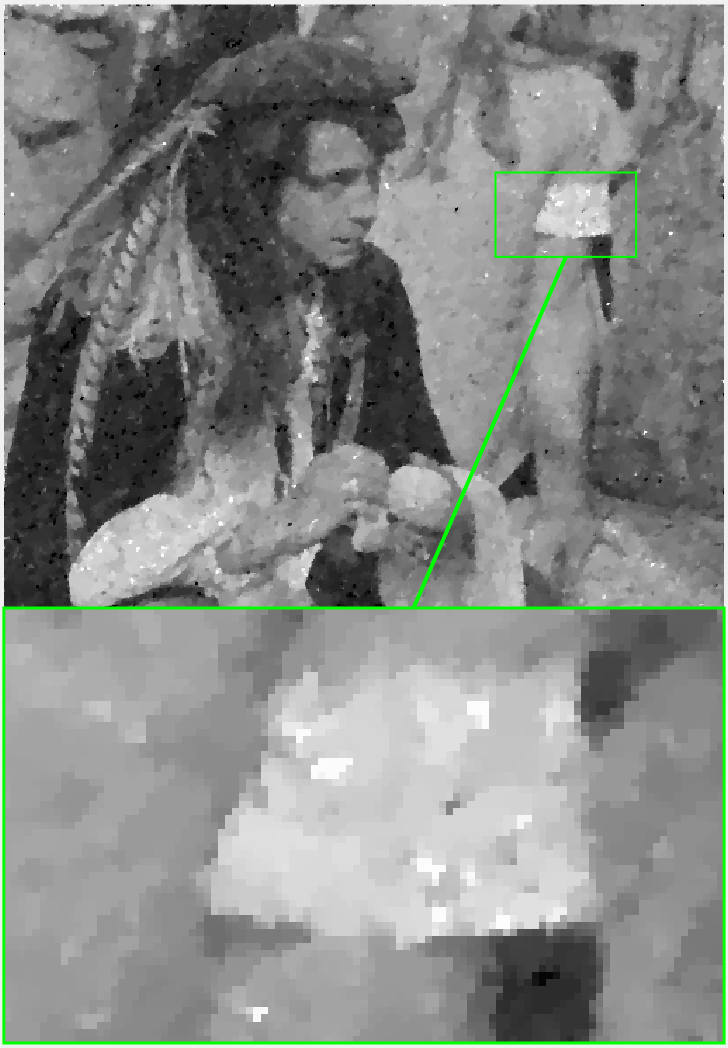}
\centerline{\footnotesize (d) DCA\cite{tao1996numerical}}
\centerline{\footnotesize PSNR:25.32}
\end{minipage}
\begin{minipage}[t]{0.28\linewidth}
\centering
\includegraphics[width=0.99\linewidth]{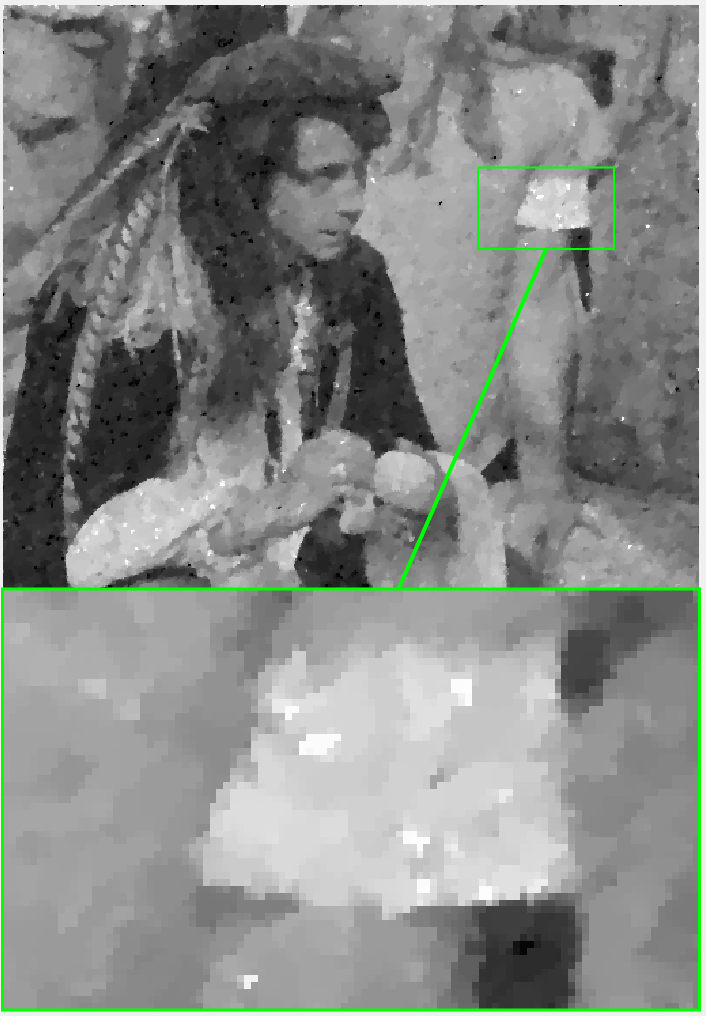}
\centerline{\footnotesize (e) nmBDCA\cite{ferreira2024boosted}}
\centerline{\footnotesize PSNR:25.53}
\end{minipage}
\begin{minipage}[t]{0.28\linewidth}
\centering
\includegraphics[width=0.99\linewidth]{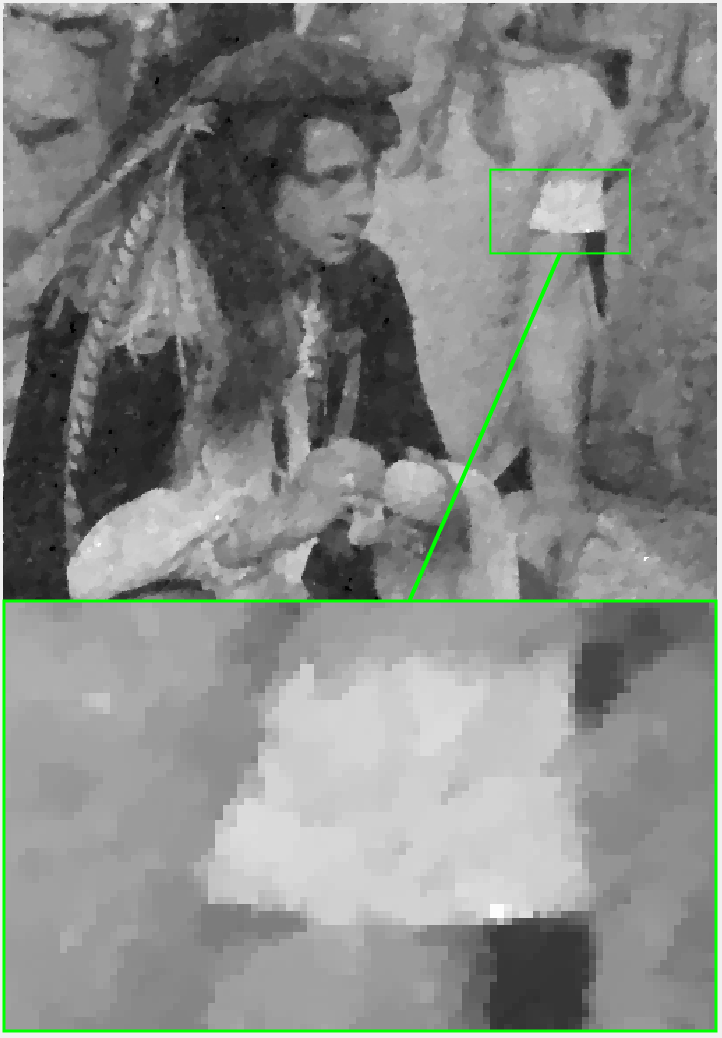}
\centerline{\footnotesize (f) Ours }
\centerline{\footnotesize PSNR:25.89}
\end{minipage}
\caption{Visual comparisons on the selected restored image, which features Cauchy noise with $\gamma = 5$. Zoom in for better visualization.}
\centering
\label{man5}
\end{figure*}

%%%  gamma 5  / bird %%%
\begin{figure*}[!ht]
\centering
\begin{minipage}[t]{0.28\linewidth}
\centering
\includegraphics[width=0.99\linewidth]{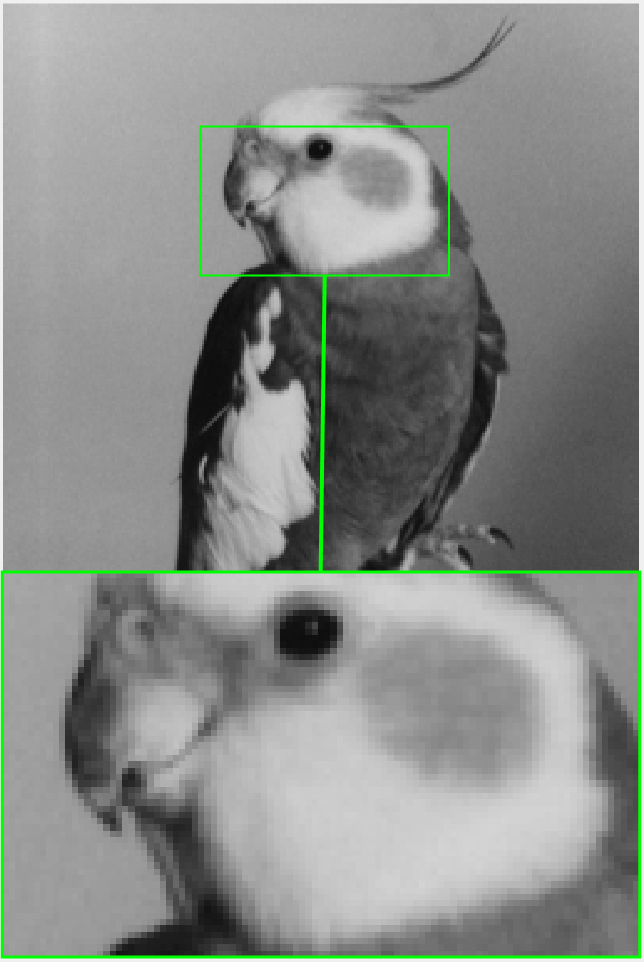}
\centerline{\footnotesize (a) Bird}
\centerline{\footnotesize Groundtruth}
\end{minipage}%
\vspace{0.1cm}
\begin{minipage}[t]{0.28\linewidth}
\centering
\includegraphics[width=0.99\linewidth]{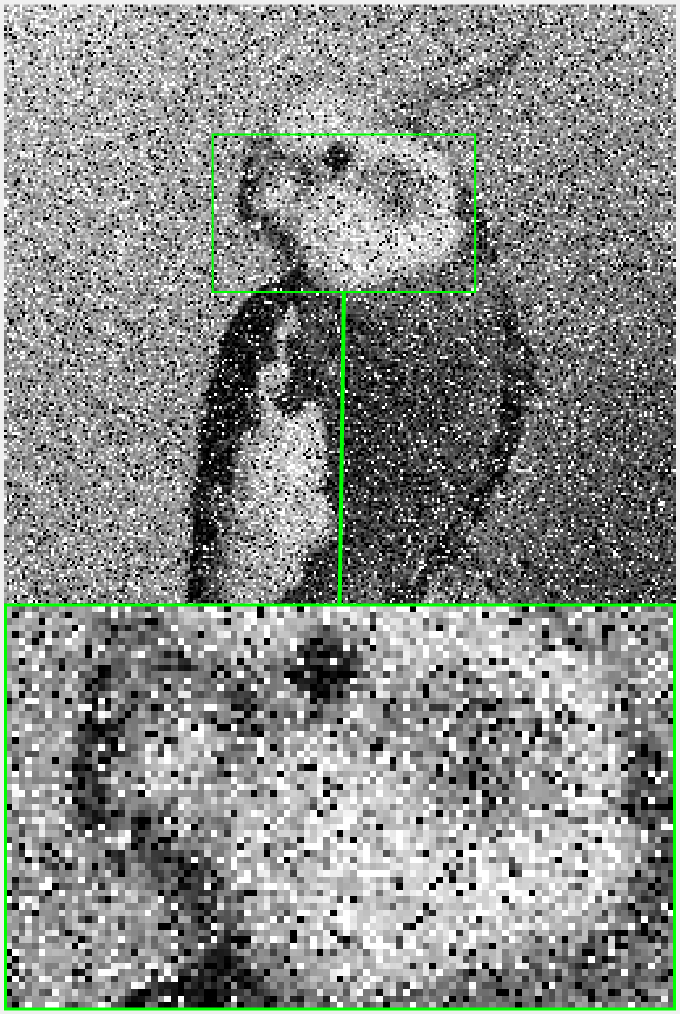}
\centerline{{\footnotesize (b) Noisy image}}
\centerline{{\footnotesize PSNR:12.70}}
\end{minipage}%
\vspace{0.1cm}
\begin{minipage}[t]{0.28\linewidth}
\centering
\includegraphics[width=0.99\linewidth]{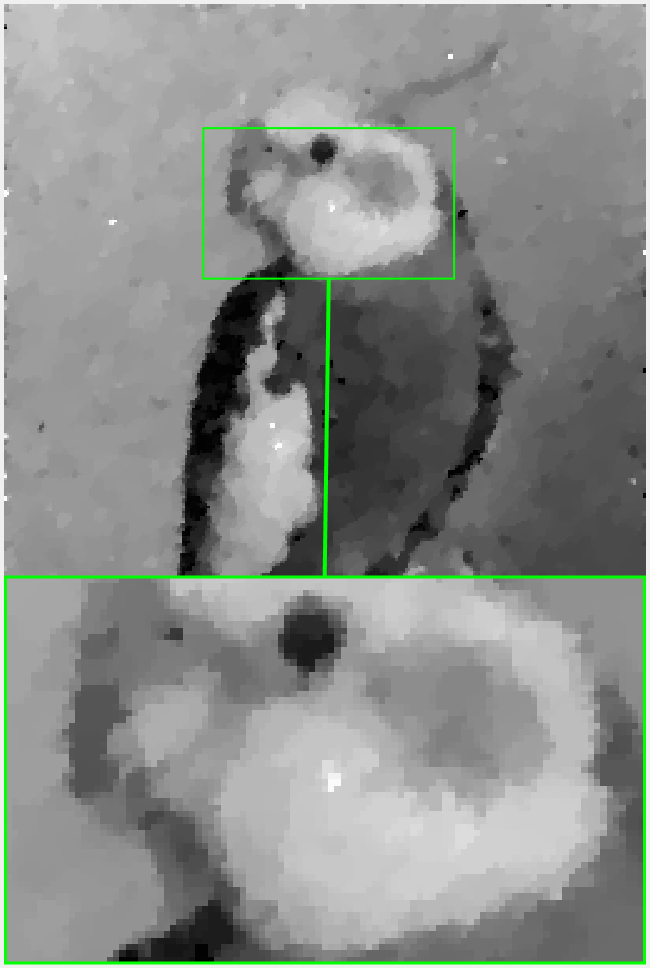}
\centerline{{\footnotesize (c) ADMM\cite{mei2018cauchy}}}
\centerline{{\footnotesize PSNR:28.34}}
\end{minipage}\\
\begin{minipage}[t]{0.28\linewidth}
\centering
\includegraphics[width=0.99\linewidth]{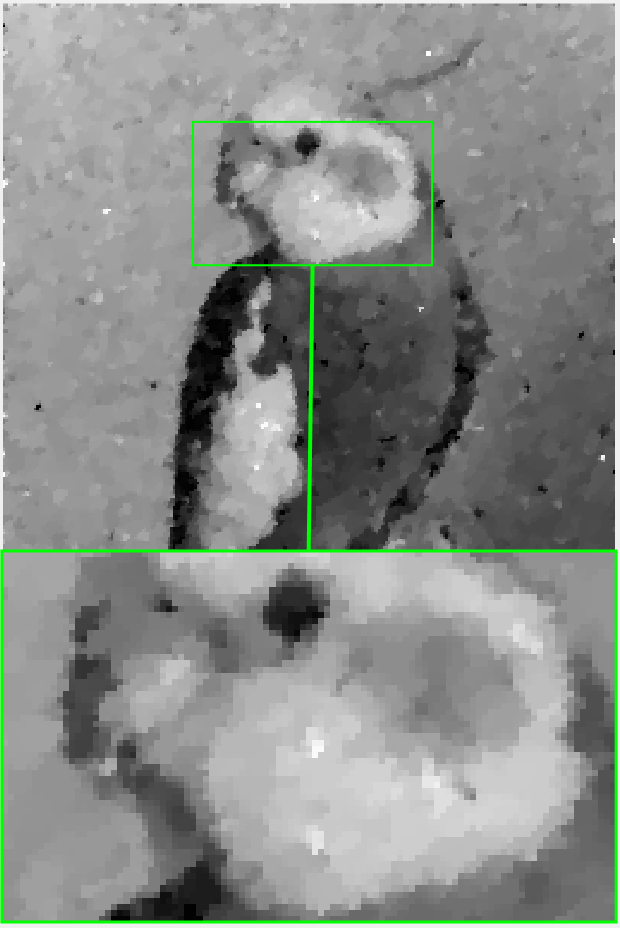}
\centerline{\footnotesize (d) DCA\cite{tao1996numerical}}
\centerline{\footnotesize PSNR:27.31}
\end{minipage}
\begin{minipage}[t]{0.28\linewidth}
\centering
\includegraphics[width=0.99\linewidth]{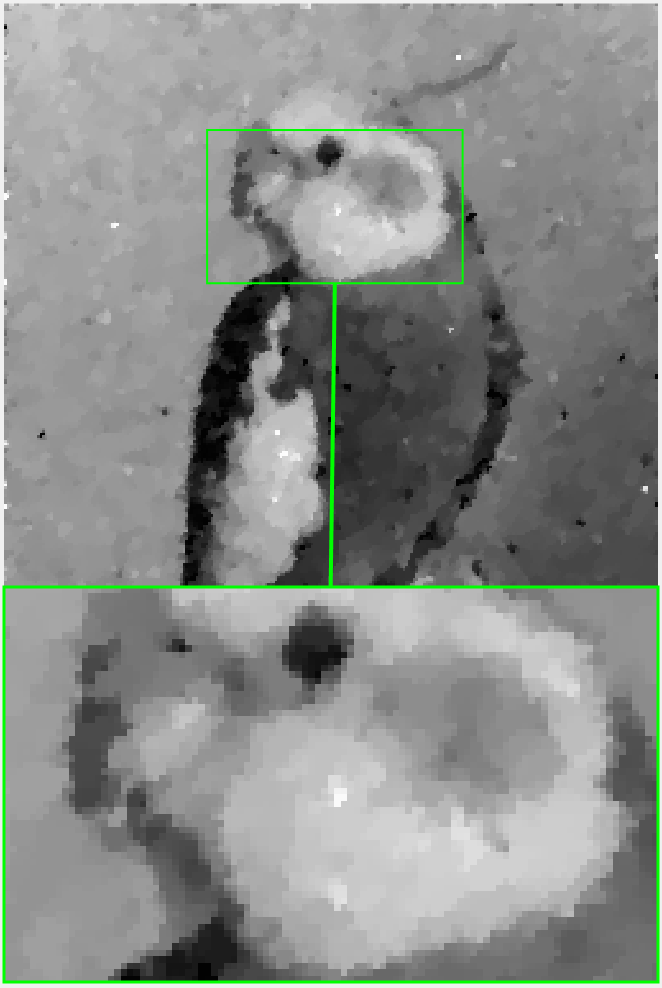}
\centerline{\footnotesize (e) nmBDCA\cite{ferreira2024boosted}}
\centerline{\footnotesize PSNR:27.79}
\end{minipage}
\begin{minipage}[t]{0.28\linewidth}
\centering
\includegraphics[width=0.99\linewidth]{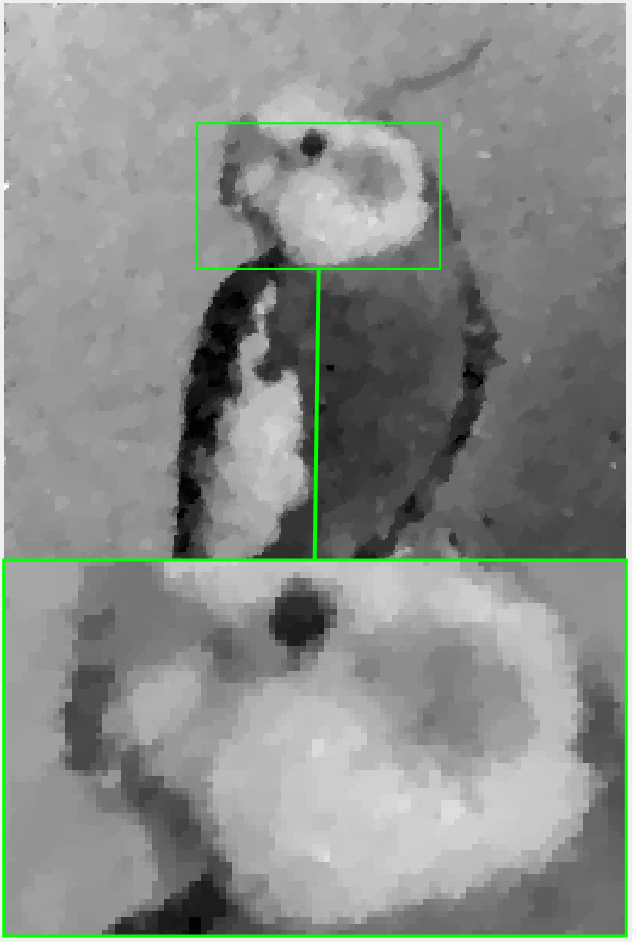}
\centerline{\footnotesize (f) Ours }
\centerline{\footnotesize PSNR:28.67}
\end{minipage}
\caption{Visual comparisons on the selected restored image, which features Cauchy noise with $\gamma = 5$. Zoom in for better visualization.}
\centering
\label{bird5}
\end{figure*}

% \clearpage
%%% energy curve of gamma =3 %%%
\begin{figure*}[!htbp]
\centering
\begin{minipage}[t]{0.49\linewidth}
\centering
\includegraphics[width=0.99\linewidth]{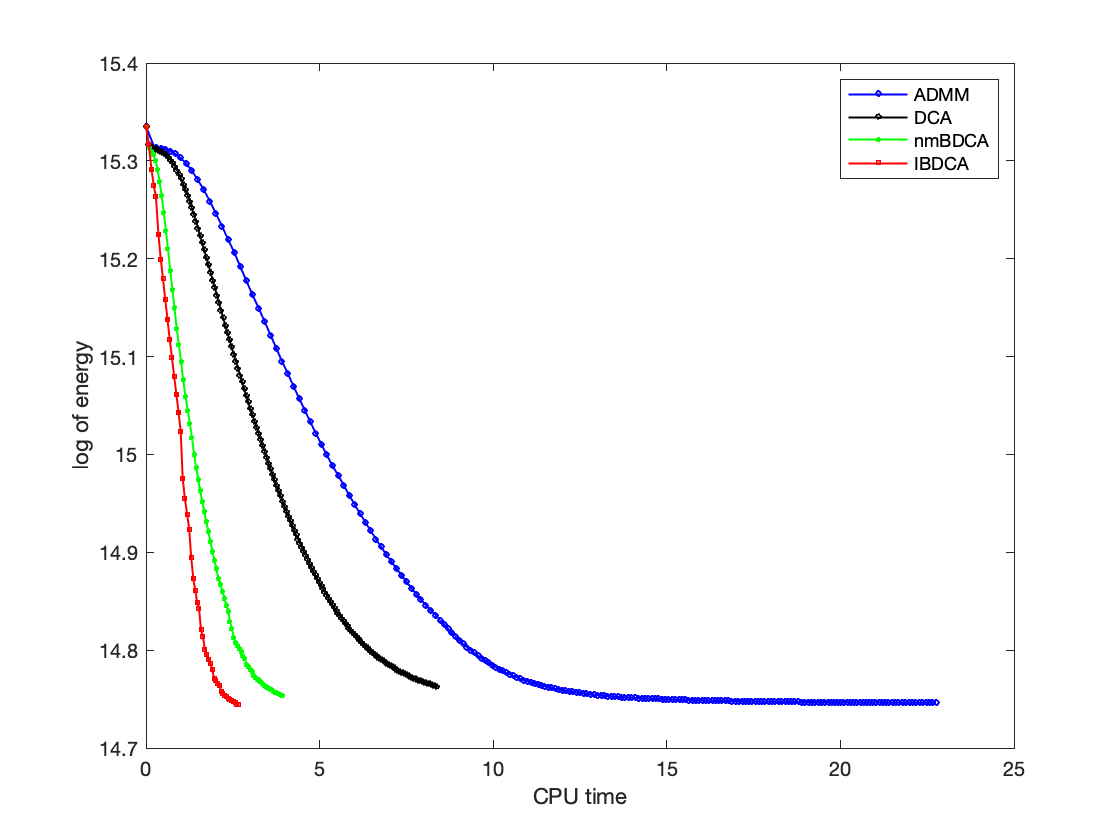}
\end{minipage}
\begin{minipage}[t]{0.49\linewidth}
\centering
\includegraphics[width=0.99\linewidth]{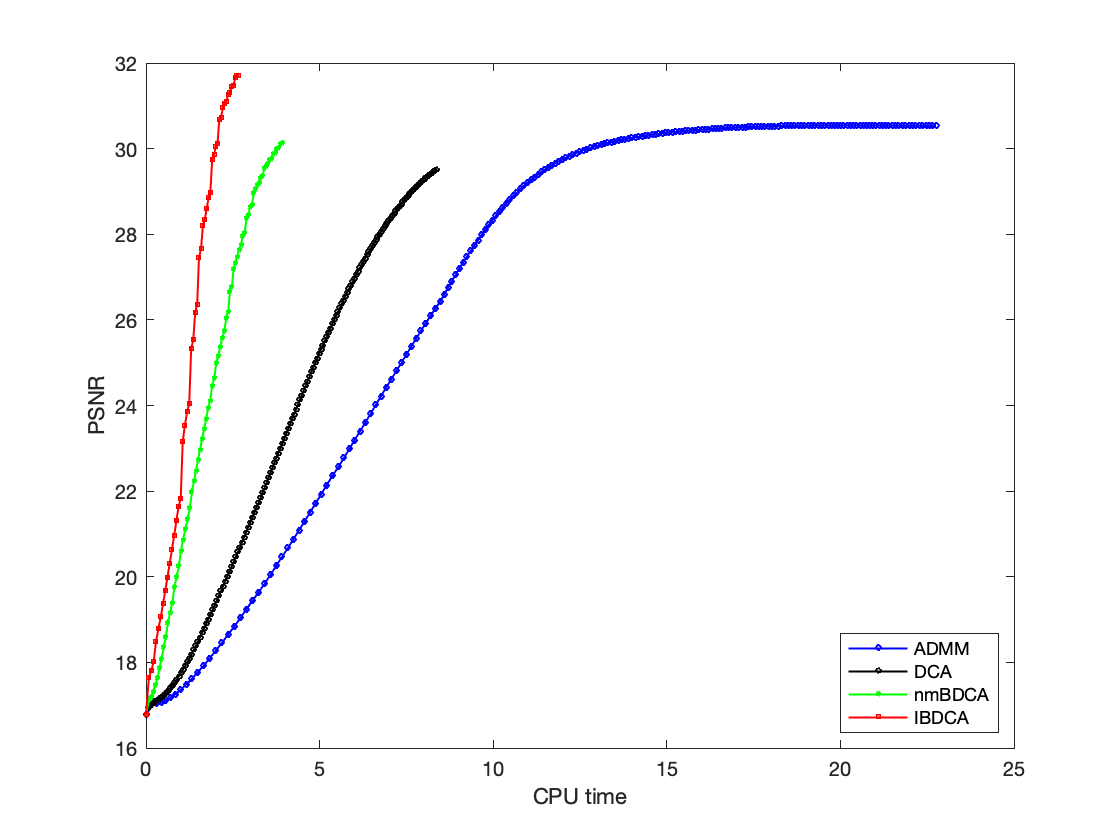}
\end{minipage}
\caption{The Energy and PSNR Curve regarding our proposed method and comparison methods of the selected image “bird” with Cauchy noise level of $\gamma=3$. Left: Energy variation over time graph, Right: PSNR variation over time graph. It can be observed that our proposed method greatly outperforms the other three comparison methods.}
\centering
\label{energy_psnr_3}
\end{figure*}

%%% energy curve of gamma =5 %%%
\begin{figure*}[!t]
\centering
\begin{minipage}[t]{0.49\linewidth}
\centering
\includegraphics[width=0.99\linewidth]{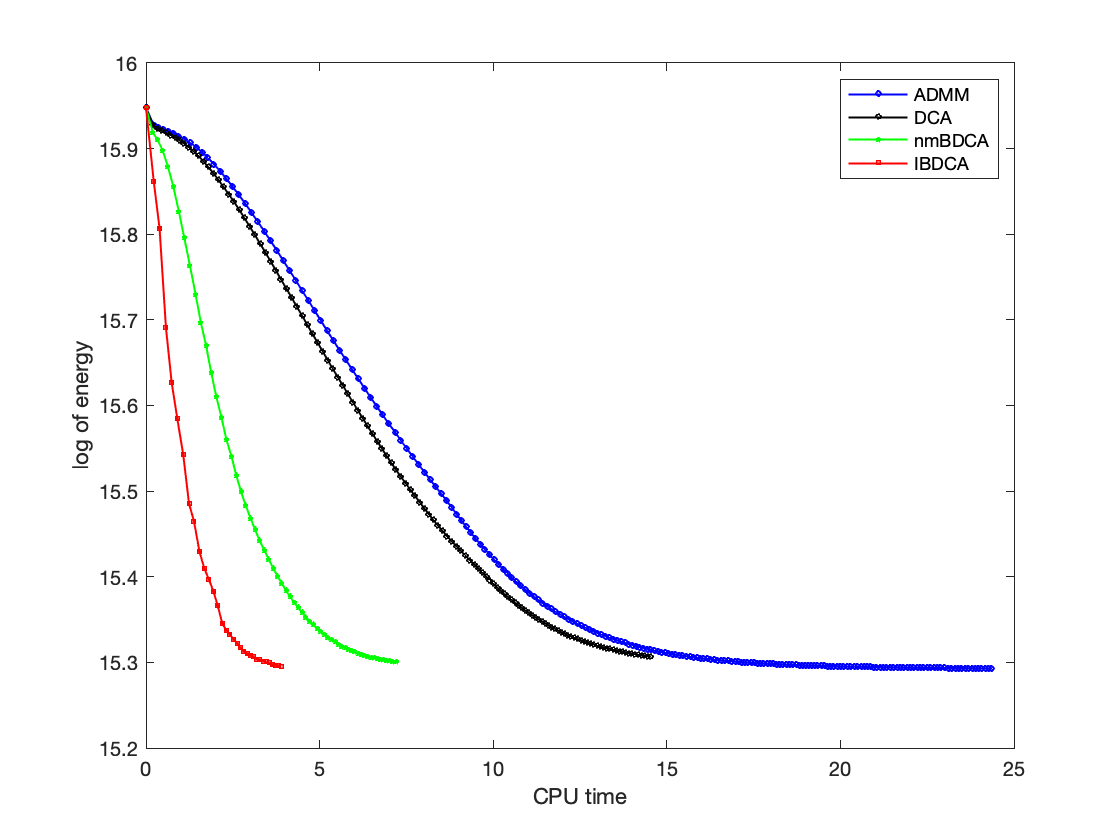}
\end{minipage}
\begin{minipage}[t]{0.49\linewidth}
\centering
\includegraphics[width=0.99\linewidth]{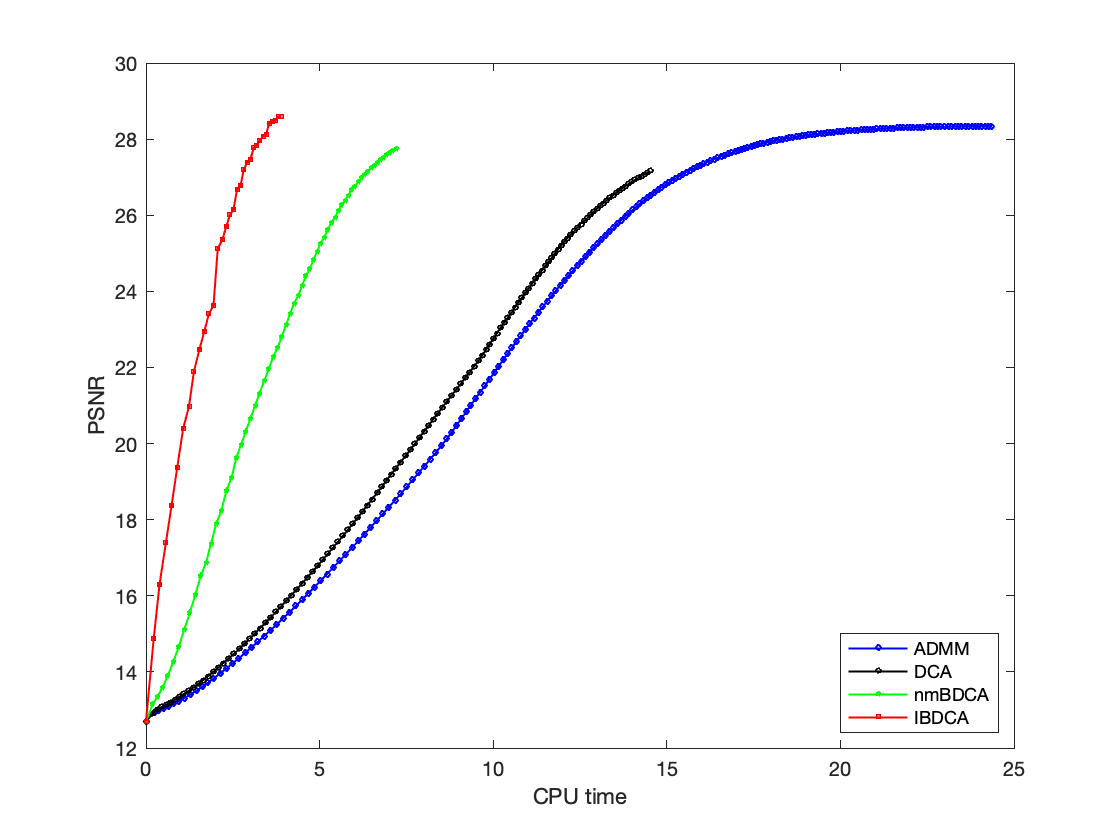}
\end{minipage}
\caption{The Energy and PSNR Curve regarding our proposed method and comparison methods of the selected image “bird” with Cauchy noise level of $\gamma=5$. Left: Energy variation over time graph, Right: PSNR variation over time graph. It can be observed that our proposed method greatly outperforms the other three comparison methods.}
\centering
\label{energy_psnr_5}
\end{figure*}

\subsection{Numerical Experiments}

 In this section, we implement our framework to image denoising tasks regarding Cauchy noise to test the effectiveness of our proposed algorithm. In the test, we employ five test images, as shown in \ref{gd}, with different sizes as our test set. We generate the noisy image by adding Cauchy noise following Cauchy distribution as follows:
$$
f = u + \gamma \frac{v_1}{v_2},
$$
where random variables $v_1$, $v_2$ follow the normal distribution $\mathcal{N}(0,1)$. To vary the Cauchy noise level, we let $\gamma$ be 3 and 5, respectively. The larger $\gamma$ is, the more noisy the image looks, making the image restoration more challenging. 
 
 We make comparisons between our proposed IBDCA and three competitive methods, including ADMM \cite{mei2018cauchy}, DCA \cite{tao1996numerical}, and nmBDCA \cite{ferreira2024boosted}. Due to the numerical experiments in \cite{ferreira2024boosted}, where nmBDCA demonstrated faster convergence and higher accuracy compared to some PPMDC (proximal point methods for DC functions,\cite{Abdellatif2006convergence,Souza2016global}), this paper does not compare with these methods. 
 
In order to enable a fair comparison, the various competing methods have been precisely adjusted and set up using their optimal parameter configurations, resulting in the most satisfactory imaging results. This permits a thorough and unbiased evaluation of the relative capabilities of the different techniques.

{For assessing numerical metrics, the peak signal-to-noise ratio (PSNR) and the relative error (ReErr) are utilized to evaluate the quality of denoising outcomes. The PSNR defined as: 
\begin{equation}
\mathrm{PSNR}=20 \log _{10} \frac{255\sqrt{m_1 m_2}}{\left\|{u}^{\star}-{u}\right\|_F},
\end{equation}
where $\|\cdot\|_F$ represents the Frobrnius norm, ${u}$ denotes the original image, and ${u}^{\star}$ is the restored image, respectively. $m_1$ and $m_2$ denote the size of the input image. The relative error is defined as follows:
 \begin{equation}
\operatorname{ReErr}=\frac{\|{u}^{\star}-{u}\|_2^2}{\|{u}\|_2^2}.
\end{equation}} 

%\subsection{Cauchy Noise Removal}
%In the task of Cauchy noise removal, we present the results of several experiments. We employ the test set consisting of five gray-scale images of different sizes to validate the effectiveness of our proposed algorithm, as shown in \cref{gd}. 

The parameters related to the proposed algorithm affect the experimental results under different levels of Cauchy noise. When $\gamma =3$, the regularization parameter that measures the trade-off between the fidelity term and the regularization term is adjusted to \( \mu = 15 \), and the initial point for all methods is set to \( u_0 = f \). The maximum of iteration is set to be 200, and the $tol$ is set to be $5 \times 10^{-4}$ for better denoised output in the stopping criterion:
$$
\frac{\left|E\left(u^k\right)-E\left(u^{k+1}\right)\right|}{E\left(u^k\right)} \leq t o l.
$$

For $\gamma =5$, we adjust $\mu = 20$, and other parameters and tolerance remain the same. It is worth noting that we need to guarantee $c \geq \mu/\gamma^2$ under any circumstances. The specific parameters for the four methods are as follows:
\begin{enumerate}
\item For ADMM, the coefficient of the penalty term should be larger than $\mu/\gamma^2$. Thus, we set this parameter as $1$ in our numerical experiments. The two subproblems in ADMM are separately solved through the Primal-Dual algorithm and Newton's method \cite{mei2018cauchy}.

\item For DCA, nmBDCA, and our proposed IBDCA, we set the parameter $c =1.10$ under a noise level of $\gamma = 5$. For the Cauchy noise level of $\gamma =3$, we adjust the value to $c = 1.83$ for better results.  

\item In IBDCA, for any $k$, the stepsize for line search is initialized as $\overline{\lambda_k}=10$, and the super line search parameter $\beta = 0.5$, and $\alpha_k=0.9(c-\frac{\mu}{\gamma^2})$.
 \item In nmBDCA, for any $k$, the stepsize for the line search is initialized as $\overline{\lambda_k}=9$ (since the line search begins at $u^k+d^k$), the line search parameter $\beta = 0.5$, $\alpha_k=0.9(c-\frac{\mu}{\gamma^2})$, and $v_k=\frac{\left\|d^k\right\|^2}{k+1}$ (\cite{ferreira2024boosted}, Example 4.6).
\end{enumerate}
\vspace{+0.5cm}
\cref{selectedimg_cauchy} shows the detail of the denoising task for our five test images with the Cauchy noise of $\gamma =3\text{,}5$. We present the PSNR(dB), RelErr, CPU time, and number of iterations corresponding to these results. From the table, it can be observed that our proposed method outperforms others in the value of PSNR and ReErr with a faster speed and fewer iterations to achieve convergence. We also provide visual comparisons of three selected images with other methods, where Fig. \ref{house}, Fig. \ref{man}, and Fig. \ref{bird} present restored images with Cauchy noise of $\gamma =3$, and Fig. \ref{house5}, Fig. \ref{man5} and Fig. \ref{bird5} show denoised outputs with Cauchy noise of $\gamma =5$. The visual comparisons indicate that our IBDCA method consistently produces the best visual effects, resulting in sharper images with enhanced structural detail and reduced noise as clearly shown in the green magnification areas. 

Figures \ref{energy_psnr_3} and \ref{energy_psnr_5} depict the relationship between the logarithm of energy and PSNR values against CPU time for our proposed method as well as three comparative methods. These illustrations underscore the benefits of our proposed IBDCA in addressing the DC problem across various levels of Cauchy noise.  These figures show that our proposed IBDCA has the lowest energy and highest PSNR value within the shortest CPU time.  

\section{Conclusion}
\label{sec:end}
We addresses a certain class of common nonsmooth DC problems and novelly develops an improved boosted DC algorithm (IBDCA). Although in these problems the search direction of BDCA in \cite{aragon2018accelerating} and \cite{nonsmoothBDCA} may not necessarily be a descent direction at the current iteration, it is a descent direction at the previous iteration. With this observation, we presented a linear search strategy and proved that, in many cases, our proposed IBDCA can achieve a greater descent magnitude than DCA with only a small amount of additional computation. Specifically, if the objective function is smooth at the current iteration, IBDCA can provide acceleration effects similar to those of BDCA in \cite{aragon2018accelerating}, IBDCA can thus be viewed as a generalization of BDCA. Under certain standard conditions, we proved that every cluster point generated by IBDCA is a critical point of the objective function and established novel monotonic descent properties. By employing the Kurdyka-Łojasiewicz property of subanalytic functions, we demonstrated global convergence and convergence rates. Future research may focus on whether similar convergence properties hold for general functions satisfying the Kurdyka-Łojasiewicz inequality. Numerical experiments and theoretical analysis have shown that IBDCA achieves a faster convergence rate, and effectively avoids the algorithm from getting trapped in locally suboptimal solutions. Furthermore, our linear search strategy can prevent increases in the objective function in certain potentially non-descent directions, making it more efficient than nmBDCA. Exploring the application of IBDCA in other fields to address nonsmooth or non-convex problems presents promising avenues for future research.

\section*{Acknowledgements}
{This work was supported in part by the National Key R\&D Program of China under Grant 2021YFE0203700, Grant ITF MHP/038/20, Grant CRF 8730063, Grant RGC 14300219, 14302920, 14301121, and CUHK Direct Grant for Research.}

\section*{Data availability}

Enquiries about data availability should be directed to the authors.

\section*{Conflict of interest}

The authors declare that they have no conflict of interest.

\clearpage
\bibliography{references}% common bib file

\end{document}